\documentclass[12pt]{amsart}
\usepackage{graphicx} 
\usepackage{amsmath, amsthm, amssymb}
\usepackage{graphicx}
\usepackage{mathbbol}
\usepackage{txfonts}
\usepackage[usenames]{color}
\usepackage{tikz}
\usepackage{color, comment}
\usepackage[colorlinks,citecolor=blue,backref=page, linkcolor=blue]{hyperref}

\counterwithin{figure}{section}

\numberwithin{equation}{section}

\newcommand{\be}{\begin{equation}}
\newcommand{\ee}{\end{equation}}

\newtheorem{thm}{Theorem}[section]
\newtheorem{ex}[thm]{Example}
\newtheorem{cor}[thm]{Corollary}   
\newtheorem{lem}[thm]{Lemma}   
\newtheorem{prop}[thm]{Proposition}
\newtheorem{defn}[thm]{Definition}

\newtheorem{rmrk}[thm]{Remark}   

\newcommand{\GHto}{\stackrel { \textrm{GH}}{\longrightarrow} }
\newcommand{\SWIFto}{\stackrel {\textrm{SWIF}}{\longrightarrow} }
\newcommand{\Fto}{\stackrel {\mathcal{F}}{\longrightarrow} }
\newcommand{\VFto}{\stackrel {\mathcal{VF}}{\longrightarrow} }

\newcommand{\mass}{{\mathbf M}}

\newcommand{\diam}{\operatorname{diam}}
\newcommand{\vol}{\operatorname{vol}}

\newcommand{\set}{\operatorname{set}}
\newcommand{\spt}{\operatorname{spt}}
\newcommand{\Lip}{\operatorname{Lip}}

\begin{document}

\title[Monotone Sequences of Metric Spaces]{Monotone Sequences of Metric Spaces with Compact Limits}

\author[R. Perales]{R. Perales}
\thanks{R. Perales' research was partially funded by NSF DMS 1612049 on {\em Geometric Compactness Theorems}.}
\address[R. Perales]{Centro de Investigaci\'on en Matem\'aticas, 
De Jalisco s/n, Valenciana, Guanajuato, Gto. Mexico. 36023}
\email{raquel.perales@cimat.mx}

\author{C. Sormani}
\thanks{C. Sormani's research was partially funded by NSF DMS 1612049 on {\em Geometric Compactness Theorems} and a PSC-CUNY grant.}
\address[C. Sormani]{CUNY Graduate Center and Lehman College}
\email{sormanic@gmail.com}

\date{}


\begin{abstract}
In this paper,
we consider a fixed metric space (possibly an oriented Riemannian manifold with boundary) with an increasing sequence of distance functions  and a uniform upper bound on diameter.   When the fixed space endowed with the pointwise limit of these distances is compact, then there is uniform and Gromov-Hausdorff (GH) convergence to this space.   When the fixed metric space also has an integral current structure  of uniformly bounded total mass (as is true for
an oriented Riemannian manifold with boundary that has a uniform bound on total volume), we prove volume preserving intrinsic flat convergence to a subset of the GH limit whose closure is the whole GH limit.   We provide a review of all notions and have a list of open questions at the end.   {\em Dedicated to Xiaochun Rong}.
\end{abstract}
\maketitle

\section{Introduction}

Our goal is to teach the notions of Gromov-Hausdorff (GH) and Sormani-Wenger intrinsic flat (SWIF) convergence while proving a new theorem that has no assumptions on curvature.  The notion of GH distance between metric spaces was first introduced by Edwards in \cite{Edwards}
and deeply explored by Gromov in 
\cite{Gromov-text} and \cite{Gromov-poly}.  
See Rong's 2010 survey \cite{Rong-survey}
for many applications of GH convergence to sequences of Riemannian manifolds with curvature bounds. 

The notion of SWIF distance between integral current spaces was 
introduced by 
Sormani and Wenger in \cite{SW-JDG} applying the
theory of Ambrosio-Kirchheim in \cite{AK} and
work of Wenger in \cite{Wenger-weak-flat}. 
Volume preserving intrinsic flat ($\mathcal{VF}$) convergence was introduced by Portegies in \cite{Portegies-evalue}.
See
Sormani's survey \cite{Sor-survey-2017} for many applications of SWIF
convergence to sequences of Riemannian manifolds with lower bounds on their scalar curvature.
See also papers by Allen, Bryden, Huang, Jauregui, Lakzian, Lee, Perales, Portegies, Sormani, and Wenger which prove SWIF and $\mathcal{VF}$ convergence theorems and present counterexamples \cite{SW-JDG}\cite{Wenger-compactness}\cite{LS2013}
\cite{AS-contrasting}\cite{AS-relating}
\cite{HLSa}\cite{APS-VADB-JDG}\cite{AP-VADB-bndry}
\cite{AB-VADB-bndry-II}\cite{Jauregui-Lee-VF}\cite{Jauregui-Perales-Portegies}. 

The work in this paper is inspired by a theorem in the appendix of \cite{HLSa} by Huang-Lee-Sormani 
(which applies to sequences with biLipschitz bounds on their distances) and a theorem
within \cite{APS-VADB-JDG} by Allen-Perales-Sormani  (which assumes only volume converging from above and distance from below but requires the limit space to be a compact smooth oriented Riemannian manifold).  
Here we will only assume the
limit space is a compact metric space but we add the hypothesis that the distances are monotone increasing.  Our result will be applied by Sormani-Tian-Yeung in \cite{STY-ex} to prove SWIF and GH convergence to the extreme limits constructed by Sormani-Tian-Wang in \cite{STW-ex}.

\begin{thm}\label{thm:Riem}
Given a compact connected Riemannian manifold, $(M^m,g_0)$,
possibly with boundary, with a monotone increasing sequence of Riemannian 
metric tensors $g_j$ such that
\be\label{eq:mono-g}
g_j(V,V)\ge g_{j-1}(V,V) \qquad \forall V\in TM
\ee
with uniform bounded diameter,
\be\label{eq:diam}
\diam_{g_j}(M)\le D_0
\qquad \forall j \in {\mathbb N}.
\ee
Then the induced length distance functions $d_j: M\times M\to [0,D_0]$
are monotone increasing and converge pointwise to
a distance function,
$d_\infty: M\times M\to [0,D_0]$
so that $(M,d_\infty)$ is
a metric space. 

If the metric space $(M, d_\infty)$ is a compact metric space,
then $d_j\to d_\infty$ uniformly and
we have Gromov-Hausdorff (GH) convergence,
\be\label{eg:GH}
(M,d_j) \GHto (M,d_\infty).
\ee

If, in addition, $M$ is an oriented manifold 
with uniform bounds on volume
and boundary volume,
\be \label{eq:vol}
\vol_j(M)\le V_0 \textrm{ and }\vol_j(\partial M)\le A_0
\qquad \forall j\in {\mathbb N}
\ee
then we have volume preserving intrinsic flat ($\mathcal{VF}$) convergence
\be \label{eq:VF-Riem}
(M,d_j, [[M]]) \VFto (M_\infty,d_\infty,T_\infty),
\ee
where $T_\infty$ is an 
integral current on $(M,d_\infty)$ such
that ``$T_\infty=[[M]]$ viewed
as an integral current on $(M,d_j)$''
and
\be \label{eq:set-Riem}
M_\infty=\set_{d_\infty}(T_\infty)\subset M \textrm{ with closure }
\overline{M}_\infty=M.
\ee
\end{thm}

In fact, we prove more general statements that do not require $M$ to be a manifold. See
Theorem~\ref{thm:GH-cmpct} within for GH convergence of a monotone sequence of metric spaces with a uniform upper bound on diameter whose distance functions converge pointwise to the distance function of a compact metric space.  See Theorem~\ref{thm:SWIF-cmpct}
and
Definition~\ref{defn:equal-current}
within for ${\mathcal{VF}}$ and SWIF convergence.  To keep the introduction simple, we do not state these results here.

Note that Sormani-Wenger already proved that any GH converging sequence satisfying (\ref{eq:vol}) has a 
subsequence converging to a SWIF limit that is a possibly empty subset of the GH limit in \cite{SW-JDG}.   In the Riemannian theorem above we use the
monotonicity to force the whole sequence to converge and to force the closure of the SWIF limit to agree with the GH limit. We will see in Example~\ref{ex:lim-cusp} within that an increasing sequence can converge to a manifold with a cusp singularity, and since a cusp singularity is not included in the definition of an integral current space, this example
has $M_\infty=\set(T_\infty)\neq M$.   

We will prove this Riemannian theorem in stages as consequences of more general results concerning metric spaces
and integral current spaces.   We will review the various notions of convergence right before applying them, so that this article may be easily read by a novice.   Note that GH, SWIF, and $\mathcal{VF}$ limit spaces are defined using distance functions and do not necessarily agree with the metric completions of limits spaces
found by taking smooth convergence of
Riemannian metric tensors away from singular sets  (see work of Allen-Sormani \cite{AS-contrasting}).

In Section~\ref{sect:dist} we prove a few
basic lemmas about metric spaces with monotone increasing sequences of distance functions including a review of the Riemannian distance functions.  In Lemma~\ref{lem:diam} we prove pointwise convergence
of the monotone increasing sequence of distance functions with a uniform upper bound on diameter.  In Example~\ref{ex:not-unif} we see that we need not have uniform convergence when the limit is noncompact.

In Section~\ref{sect:GH} 
we review the definition of Gromov-Hausdorff convergence (Definition~\ref{defn:GH}).  We then prove the GH and uniform convergence part of
Theorem~\ref{thm:Riem}.  In fact, we prove a more general theorem concerning a metric space with a monotone increasing sequence of distance functions that
converge pointwise to a compact limit in Theorem~\ref{thm:GH-cmpct}.  We show that the sequence converges uniformly and in the GH sense.
We prove this theorem using different techniques than would have been used by Gromov in order to prepare for the SWIF convergence part of the paper. 
In particular, we construct a common metric space for the converging sequence that has special properties in
Proposition~\ref{prop:common-Z} and Lemma~\ref{lem:W-to-X}. In Remark~\ref{rmrk:not-GH} we explain that Example~\ref{ex:not-unif} does not have a Gromov-Hausdorff limit.

In Section~\ref{sect:SWIF} we prove the SWIF convergence
part of Theorem~\ref{thm:Riem}.   In fact, we prove a more general SWIF convergence theorem concerning the convergence of an integral current space with increasing distance functions [Theorem~\ref{thm:SWIF-cmpct}]. We begin by reviewing  De Giorgi's notion of a Lipschitz tuple \cite{DeGiorgi95}, the notion of a Lipschitz chart, and Ambrosio-Kirchheim's notion of an integral current and weak convergence \cite{AK}, proving lemmas about how these objects behave for monotone increasing sequences of distances.  
We review Sormani-Wenger's definition of an integral current space, including $(M,d,[[M]])$, and their definition of intrinsic flat distance.   We then state Theorem~\ref{thm:SWIF-cmpct} and apply it to complete the proof of Theorem~\ref{thm:Riem} in Section~\ref{sect:Riem}. Finally we prove Theorem~\ref{thm:SWIF-cmpct}, first finding a
converging subsequence and a limit current structure in Proposition~\ref{prop:Tinfty} and completing the proof in Section~\ref{sect:pf-SWIF} applying Proposition~\ref{prop:TaTb} to estimate the SWIF distance.

In Section~\ref{sect:Open}
we state open problems.   

The authors would like to thank Wenchuan Tian and Changliang Wang for inspiring discussions of examples that motivated this paper.  We would also like to thank the referee and students, Wai Ho Yeung and Kevin Bui, for careful reading of various parts of the first version of this paper and helpful feedback.

\section{Pointwise Convergence of Monotone Increasing Distance Functions}
\label{sect:dist}

\begin{defn}\label{defn:metric-space}
A metric space, $(X,d)$, is a collection of
points, $X$, paired with a distance function
$d: X\times X\to [0,\infty)$ which is
definite,
\be
d(x,y)=0 \,\,\iff\,\, x=y,
\ee
symmetric, 
\be
d(x,y)=d(y,x)\quad \forall x,y \in X,
\ee
and satisfies the triangle inequality,
\be
d(x,y)\le d(x,z)+d(z,y) \qquad \forall x,y,z\in X.
\ee
\end{defn}

Here we consider metric spaces $(X,d)$ of finite diameter,
\be
\diam_d(X)=\sup\{d(x,y):\,x,y\in X\}<\infty.
\ee
We also need the following definition:

\begin{defn}\label{defn:mono}
We say that $d_j:X\times X\to [0,\infty)$ is a monotone increasing sequence of distance functions on $X$, if for all $j\in {\mathbb{N}}$ one has that $(X,d_j)$ is a metric space and
\be
d_{j+1}(x,y)\ge d_j(x,y) \quad \forall x,y\in X.
\ee
\end{defn}

\subsection{Riemannian Distances}

Recall that a Riemannian manifold, $(M,g)$,
has a natural metric space, $(M, d_g)$, defined
using the induced Riemannian distance as follows.
The lengths of piecewise smooth
curves, $C:[0,1]\to M$, are defined by
\be\label{eq:length}
L_g(C)=\int_0^1 g(C'(t), C'(t))^{1/2} \, dt
\ee
and the Riemannian distance between points
is defined by
\be\label{eq:d_g}
d_g(x,y)=\inf \{ L_g(C) \}
\ee
where the infimum is taken over all piecewise smooth curves, $C:[0,1]\to M$, such that
$C(0)=x$ and $C(1)=y$.   The diameter of $(M, d_g)$ is then defined as
\be\label{eq:diam-defn}
\diam_{g}(M)=\diam_{d_g}(M)=\sup \{d_g(x,y):\, x,y\in M\}.
\ee

\begin{lem}\label{lem:Riem-ineq}
If $M$ is a Riemannian manifold with
two Riemannian metric tensors $g_a$
and $g_b$ such that
\be
g_a(V,V)\le g_b(V,V) \quad \forall V\in TM
\ee 
then the corresponding induced
Riemannian distances $d_a=d_{g_a}$
and $d_b=d_{g_b}$ satisfy
\be
d_a(x,y) \le d_b(x,y) \qquad \forall x,y\in X.
\ee 
\end{lem}

The proof of Lemma~\ref{lem:Riem-ineq} is an easy exercise applying
\ref{eq:length} and \ref{eq:d_g}.

\begin{lem}\label{lem:Riem-mono}
Consider a fixed Riemannian manifold, $(M,g_0)$,
with a sequence of 
functions, $h_j: M \to [0,\infty)$, such that
the gradients have magnitudes
satisfying,
\be \label{eq:magnitude}
|\nabla h_j|_{g_0}\le | \nabla h_{j+1}|_{g_0}
\ee
everywhere on $M$,
and have the
same direction,
\be \label{eq:direction}
\frac{\nabla h_j}{\,\,\,\,|\nabla h_j|_{g_0}}=\frac{\nabla h_{j+1}}{\,\,\,\,|\nabla h_{j+1}|_{g_0}}
\ee
whenever $|\nabla h_j|_{g_0}>0$.
Then the
sequence of Riemannian metrics 
\be
g_j = g_0 + dh_j^2 
\ee
induced by the graphs of the functions, $h_j$,
is a 
monotone increasing sequence
of metric tensors, $g_j$ as in
as in (\ref{eq:mono-g}), and defines
a monotone
increasing sequence of distance functions as
in Definition~\ref{defn:mono}.
\end{lem}

\begin{proof}
By Lemma~\ref{lem:Riem-ineq},
we need only show (\ref{eq:mono-g}).  Recall that
\be
g_{j}(V,V)=
g_0(V,V) + dh_j^2(V,V)
=
g_0(V,V) + g_0(\nabla h_j, V)^2.
\ee
At points where $|\nabla h_j|_{g_0}=0$, we thus have
\be
g_{j}(V,V)=
g_0(V,V)+0
\le
g_0(V,V) + dh_{j+1}^2(V,V)=g_{j+1}(V,V).
\ee
At points where $|\nabla h_j|_{g_0}>0$, the gradients have
the same direction
so
\begin{eqnarray}
g_{j}(V,V)&=&
g_0(V,V) + g_0(\nabla h_j, V)^2\\
&=&
g_0(V,V) + 
g_0(\tfrac{\nabla h_j}{|\nabla h_j|}, V)^2\, |\nabla h_j|^2\\
&=&
g_0(V,V) + 
g_0(\tfrac{\nabla h_{j+1}}{|\nabla h_{j+1}|}, V)^2\, |\nabla h_j|^2\\
&\le &
g_0(V,V) +
g_0(\tfrac{\nabla h_{j+1}}{|\nabla h_{j+1}|}, V)^2\,
|\nabla h_{j+1}|^2\\
&=&
g_0(V,V) + dh_{j+1}^2(V,V)=g_{j+1}(V,V).
\end{eqnarray}
This completes the proof
that (\ref{eq:mono-g}) holds at all points in 
$M$.
\end{proof}

\subsection{Pointwise convergence}

\begin{lem}\label{lem:diam}
Suppose that $d_j:X\times X\to [0,\infty)$ is a monotone increasing sequence of distance functions on $X$
as in Definition~\ref{defn:mono}
and that there is
a uniform upper bound on diameter,
\be
\diam_j(X)=\sup_{x,y\in X}d_j(x,y) \le D_0\in (0,\infty)
\ee
then there is a limit distance function:
\be\label{eq:diam-0}
d_\infty(x,y)=\lim_{j\to \infty}d_j(x,y) =
\sup_{j\in {\mathbb N}} d_j(x,y) \in [0,D_0]
\ee
and $(X, d_\infty)$ is a metric
space with diameter $\diam_\infty(X)\le D_0$.
\end{lem}

\begin{proof}
We know the pointwise limit function $d_\infty: X\times X\to [0,D_0]$ exists and satisfies (\ref{eq:diam-0}) by the monotone convergence theorem for sequences of real numbers.   We can confirm that $d_\infty$ is positive definite as follows:
\be d_{\infty}(x,x)=\lim_{j\to \infty} d_j(x,x)=0
\ee
and
\be
d_{\infty}(x,y)=0
\implies
d_j(x,y)\le 0
\implies x=y.
\ee
It is symmetric for all $x,y\in X$ by
\be
d_\infty(x,y)=
\lim_{j\to \infty} d_j(x,y)
=\lim_{j\to \infty} d_j(y,x)=
d_\infty(y,x) \quad \forall x,y,\in X.
\ee
It satisfies the triangle inequality 
for any $x,y,z\in X$
by
\begin{eqnarray}
d_\infty(x,y)&=&
\lim_{j\to \infty} d_j(x,y)\\
&\le&\lim_{j\to \infty} ( d_j(x,z)+d_j(z,y))\\
&=&\lim_{j\to \infty} d_j(x,z)+
\lim_{j\to \infty} d_j(z,y)\\
&=& d_\infty(x,z)+d_\infty(z,y) \quad \forall x,y,z\in X.
\end{eqnarray}

\end{proof}

\subsection{Example without Uniform Convergence}

Even if $(X, d_j)$ is a monotone increasing sequence of compact metric spaces with a uniform upper bound on diameter,
pointwise convergence does not imply uniform convergence.   See Example~\ref{ex:not-unif}
depicted in Figure~\ref{fig:not-unif}.

\begin{figure}[h] 
   \centering
\includegraphics[width=.8\textwidth]{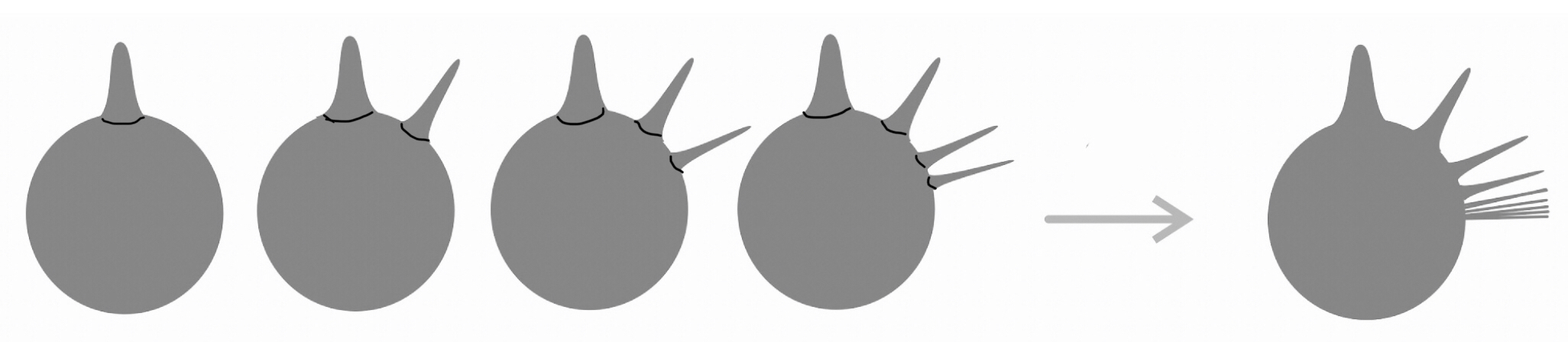} 
\caption{This figure depicts Example~\ref{ex:not-unif}.
   }
   \label{fig:not-unif}
\end{figure}

\begin{ex}\label{ex:not-unif}
Here we adapt an example from Sormani-Wenger \cite{SW-JDG}
as in Figure~\ref{fig:not-unif}.
Consider a sequence of distinct points $p_j \to p_\infty$ lying on the equator in the standard round sphere $({\mathbb S}^2, g_{\mathbb S^2})$.   Let $r_j >0$ be decreasing radii such that the balls $B_{g_{\mathbb S^2}}(p_j,r_j)$ are disjoint and do not contain $p_\infty$.  
We define a sequence of metric tensors 
\be
g_j=g_{{\mathbb S}^2}+dh_j^2
\ee
 inductively starting with
constant $h_0=0$ and taking
\be \label{eq:start-ind}
h_{j+1}(x) = 
\begin{cases}
        h_{j}(x) & \text{if } x\in {\mathbb S}^2\setminus B_{g_{\mathbb S^2}}(p_{j+1},r_{j+1})\\
        H(d_{\mathbb S^2}(p_{j+1},x)/r_{j+1}) & \text{if } x\in B_{g_{\mathbb S^2}}(p_{j+1},r_{j+1})
\end{cases}
\ee
where $H:(-1,1)\to [0,H_0]$ is a fixed smooth even
function of compact support such that $H(0)=H_0>0$
and $H'(r)\le 0$ for $r\ge 0$.  

This implies that, 
for $x$ outside of the disjoint balls about $p_0$ to $p_{i}$,
we have
\be
h_{i+1}(x)=h_i(x)=...=h_0(x)=0. 
\ee
Since the balls are disjoint, 
for all $x\in B_{g_{\mathbb S^2}}(p_j,r_j)$,
$x$ is not in the previous balls, so 
$h_j(x)=0$, so $\nabla h_j=0$.   Elsewhere, by (\ref{eq:start-ind}) we have
$h_j=h_{j+1}$ so $\nabla h_j=\nabla h_{j+1}$.  Thus we can apply Lemma~\ref{lem:Riem-mono}
to see that we have a monotone increasing sequence
of metric tensors.

Note that we have added increasingly thin wells in place of each of the balls.  For all $i\le j$, the $g_j$ radial distance
from $p_i$ to any $q_i\in \partial B_{g_{\mathbb S^2}}(p_j,r_j)$,
\be
d_{g_j}(p_i,q_i)=\int_0^{r_i}
\sqrt{1+(H'(r/r_i) (1/r_i))^2} \, dr.
\ee
Since $a\le \sqrt{1+a^2}\le 1+a$ we have
\be\label{eq:pt-well}
H_0 \le d_{g_j}(p_i,q)\le r_i+H_0 \quad \forall i\le j.
\ee
Since any point in $({\mathbb S}^2,g_j)$ outside of the wells can be joined to a pole travelling a distance
$\le \pi/2$, and poles are a distance $\pi$ apart, we can apply the triangle inequality 
to see that
\be
\diam_{g_j}({\mathbb S}^2)\le (H_0+r_1+\pi/2)+ (H_0+r_1+\pi/2)+\pi.
\ee
This uniform upper bound and Lemma~\ref{lem:diam}
implies there is a pointwise limit, $d_j \to d_\infty$.
Taking the limit of (\ref{eq:pt-well}) we have,
\be
H_0\le d_\infty(p_i,p_\infty)  \quad \forall i\in {\mathbb N}.
\ee
However, for $i>j$ sufficiently large
\be
d_{g_j}(p_i,p_\infty)
\le d_{\mathbb S^2}(p_i,p_\infty)
<H_0/2 
\ee
so there is no uniform convergence. In fact, the limit distance function $d_\infty$ has a sequence of disjoint $g_j$ balls centered at the $p_j$ of radius $H_0$, so $({\mathbb S}^2,d_\infty)$ is not compact.
\end{ex}

\section{GH Convergence to Compact Limits}
\label{sect:GH}

Recall the following definition of the
Gromov-Hausdorff distance first introduced by 
Edwards in \cite{Edwards} and then rediscovered and
studied extensively by Gromov in \cite{Gromov-text}.

\begin{defn}\label{defn:GH}
Given a pair of metric spaces, $(X_a,d_a)$ and $(X_b,d_b)$, the Gromov-Hausdorff distance
between them is
\be
d_{GH}\left((X_a,d_a),(X_b,d_b)\right)=
\inf\{ d_H^Z\left(f_a(X_a), f_b(X_b)\right)\}
\ee
where the infimum is taken over all common metric spaces, $(Z,d_Z)$, and over all distance preserving maps,
\be
f_a: (X_a,d_a)\to (Z,d_Z)
\textrm{ and } f_b: (X_b,d_b)\to (Z,d_Z).
\ee
Here $d_H^Z$ denotes
the Hausdorff distance,
\be
d_H^Z\left(A,B\right)
=\inf \{R \,:\, A\subset T_R(B)
\textrm{ and } B\subset T_R(A)\}
\ee
between subsets $A,B \subset Z$ which is defined
using tubular neighborhoods of a given radius $R$,
\be
T_R(A)=\{z\in Z\,:\, \exists p\in A
\,s.t.\, d_Z(p,z)<R\}.
\ee
\end{defn} 

In this section we prove the following simple theorem:

\begin{thm}\label{thm:GH-cmpct} 
If $(X,d_j)$ is a monotone increasing sequence as in Definition~\ref{defn:mono} with a uniform upper bound $D_0$ on diameter, and if $d_j$ converges pointwise to $d_\infty$ as in Lemma~\ref{lem:diam},
and if $(X, d_\infty)$
is a compact metric space, 
then $d_j$ converge uniformly to $d_\infty$
and we have Gromov-Hausdorff convergence 
as well.  
\end{thm}

Although this GH Convergence theorem can be proven quite easily, we will prove it through a sequence of lemmas 
and propositions that we will apply again to prove
SWIF convergence later.

\subsection{Distance on the Product Space}

We introduce the following standard definition:

\begin{defn}\label{defn:dsum}
Given a pair of metric spaces,
$(X_a,d_a)$ and $(X_b,d_b)$, we
define the taxi product metric space
$(X_a\times X_b, d_{sum,a,b})$
where
\be\label{eq:dsum}
d_{sum,a,b}((x_1,y_1),(x_2,y_2))=d_a(x_1,x_2)+d_b(y_1,y_2)
\ee
for any $x_1,x_2\in X_a$ and any
$y_1,y_2\in X_b$.
\end{defn}

\begin{lem}\label{lem:cpct-dbb}
If $(X,d_a)$ and $(X,d_b)$ are compact then
$(X\times X, d_{sum,a,b})$ is also compact.
\end{lem}

The proof is an exercise for the reader.

\begin{lem}\label{lem:da-Lip}
Given two metric spaces
$(X,d_a)$ and $(X,d_b)$,
if 
\be \label{eq:ab-Lip}
d_a(x,y) \le d_b(x,y)\quad \forall x,y\in X
\ee
then the function
$d_a: X\times X \to [0,\infty)$ is $1$-Lipschitz with respect to the distance
$d_{sum,b,b}$ on $X\times X$ given in
Definition~\ref{defn:dsum}.
That is,
\be
|d_a(x_1,y_1)-d_a(x_2,y_2)|
\le d_{sum,b,b}((x_1,y_1),(x_2,y_2)) \qquad \forall (x_i,y_i)\in X\times X.
\ee  
\end{lem}

\begin{proof}
By the triangle inequality
and then by (\ref{eq:ab-Lip})
we have
\begin{eqnarray}
|d_a(x_1,y_1)-d_a(x_2,y_2)| &\le& d_a(x_1,x_2)+d_a(y_1,y_2)\\
&\le&
d_b(x_1,x_2)+d_b(y_1,y_2)
\end{eqnarray}
which gives our claim by
(\ref{eq:dsum})
\end{proof}

\subsection{Uniform Convergence}

\begin{prop}\label{prop:unif-cmpct}
Under the hypotheses of Theorem~\ref{thm:GH-cmpct},
the distance functions, $d_j$, converge uniformly
to $d_\infty$ on $X\times X$. 
\end{prop}

\begin{proof}
By Lemma~\ref{lem:diam}, we have $d_j\le d_\infty$ as
functions on $X\times X$. By Lemma~\ref{lem:da-Lip}, the maps
$d_j: X\times X \to [0,D_0]$ are bounded $1$-Lipschitz functions on the  metric space $(X\times X, d_{sum,\infty,\infty})$,
which is compact by hypothesis and Lemma~\ref{lem:cpct-dbb}.

By the Arzela-Ascoli Theorem, there is a subsequence, $d_{j_k}:X\times X\to [0,D_0]$, which
converges uniformly to a limit function.   However, we already have a pointwise limit, so that limit function equals $d_\infty: X\times X \to [0,D_0]$.  
That is,
$
\forall \epsilon>0 \,\,\exists N_\epsilon \in \mathbb N
\,\,s.t.\,\, \forall k\ge N_\epsilon$
\be
d_\infty(x,y)-\epsilon < d_{j_k}(x,y) < d_\infty(x,y)+\epsilon
\qquad \forall x,y\in X\times X.
\ee
By the assumption of monotonicity, for all $j\ge j_k\ge j_{N_\epsilon}$,
\be
d_\infty(x,y)-\epsilon < d_{j_k}(x,y) \le d_j(x,y)\le d_\infty(x,y)
\qquad \forall x,y\in X\times X,
\ee
showing that the whole sequence $d_j$ uniformly converges to $d_\infty$.
\end{proof}

\subsection{Constructing a Common Metric Space}

Recall that the definition of the Gromov-Hausdorff distance involves a common metric space $Z$ and distance preserving maps [Definition~\ref{defn:GH}].
The following lemma is not part of the original proof by Gromov that uniform convergence implies Gromov-Hausdorff convergence.   Here we create a different common metric space $Z$ that allows us to prove
intrinsic flat convergence later as well.   It is a simplification of the $Z$ constructed by Allen-Perales-Sormani in \cite{APS-VADB-JDG}.

\begin{prop}\label{prop:Zab}
Given two metric spaces
$(X,d_a)$ and $(X,d_b)$
and an $\epsilon>0$ such that
\be \label{eq:ab-Zab}
d_b(x,y)-\epsilon \le d_a(x,y) \le d_b(x,y)\quad \forall x,y\in X
\ee
then there exists a metric
space 
\be
(Z_{a,b}=X\times [0,h], d_{Z_{a,b}})
\textrm{
where } h=\epsilon/2,
\ee
with the distance 
between $z_1=(x_1,t_1)$
and $z_2=(x_2,t_2)$ defined by
\be \label{eq:d_Z}
d_{Z_{a,b}}(z_1,z_2)=
\min\{d_{taxi_{b}}(z_1,z_2), (h-t_1)+(h-t_2)+d_a(x_1,x_2)\}.
\ee
\be
\textrm{ where }
\quad
d_{taxi_{b}}(z_1,z_2)=
d_b(x_1,x_2)+|t_1-t_2|
\ee
so that the identity map, 
\be
id_{a,b}: (Z_{a,b},d_{taxi_{b}})
\to (Z_{a,b},d_{Z_{a,b}})
\ee
is $1$-Lipschitz and
there are distance preserving maps:
\begin{eqnarray}
&f_a=f_{a,(a,b)}:(X, d_a) \to Z_{a,b}&
\textrm{ s.t. } f_a(x)=(x,h),\\
&f_b=f_{b,(a,b)}:(X, d_b) \to Z_{a,b}&
\textrm{ s.t. } f_b(x)=(x,0),
\end{eqnarray}
so that
\be\label{eq-GHestimate}
d_{GH}((X,d_a), (X,d_b))
\le d_H^{Z_{a,b}}(f_a(X),f_b(X))
\le h.
\ee
\end{prop}

The intuition behind this construction is
that we are taking the taxi product of $(X,d_b)$ with the interval $([0,h],d_{\mathbb R})$
and then gluing $(X,d_a)$ to this taxi product at
$t=h$ which gives possibly shorter distances.   See Figure~\ref{fig:Zab}.

\begin{figure}[h] 
   \centering
   \includegraphics[width=.8 \textwidth]{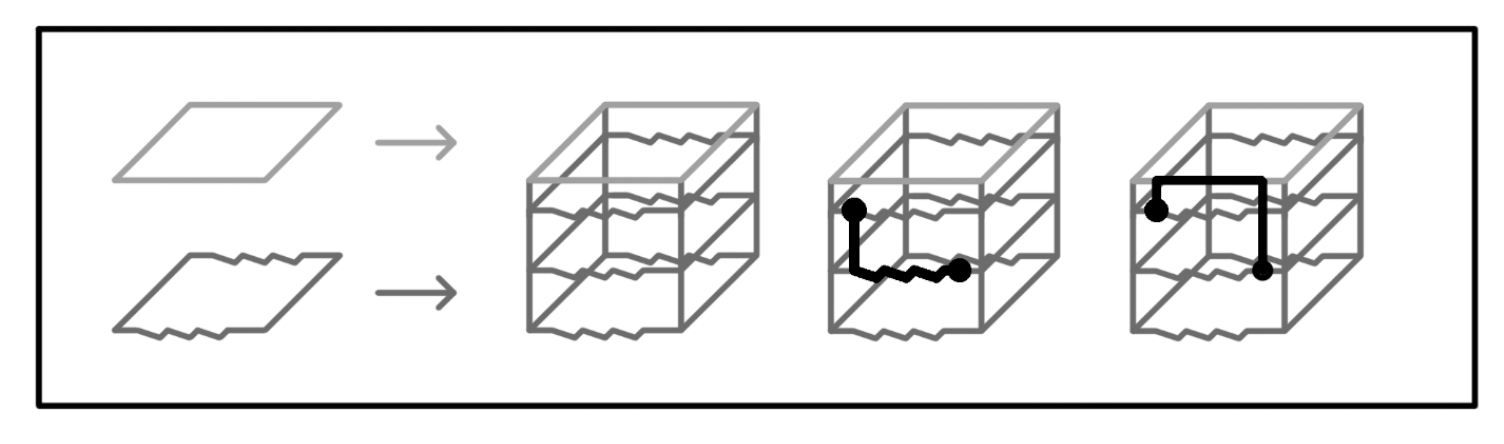} 
\caption{On the left we see
$(X_a,d_a)$ above $(X_b,d_b)$ mapping into
$(Z_{a,b},d_{Z_{a,b}})$ and on the right we see two paths between points in $Z_{a,b}$ where the first achieves
$d_{taxi_b}$ and the second
takes a short cut through the image of $X_a$.
The shorter of the two paths achieves the minimum in the definition of $ d_{Z_{a,b}}$.  }
\label{fig:Zab}
\end{figure}

\begin{proof}
First we confirm that $d_{Z_{a,b}}$ is a metric.  It is clearly definite and symmetric. To see that it satisfies the triangle inequality we have several cases.
Let $z_i=(x_i,t_i) \in Z_{a,b}$, $i=1,2,3$ and 
\be
D_{13}=d_{Z_{a,b}}(z_1,z_2) + d_{Z_{a,b}}(z_2, z_3).
\ee

\noindent
{\em{Case I:}}
Assume that 
$
d_{Z_{a,b}}(z_1,z_2)=
d_{taxi_{b}}(z_1,z_2)
$ and $
d_{Z_{a,b}}(z_2, z_3)= d_{taxi_{b}}(z_2,z_3)
$
Then since $d_{taxi_{b}}$
is a distance function, and thus satisfies the triangle inequality, and by definition of $d_{Z_{a,b}}$, we get 
\be
D_{13} =
d_{taxi_{b}}(z_1,z_2) +d_{taxi_{b}}(z_2,z_3)
\geq 
d_{taxi_{b}}(z_1,z_3) \geq 
d_{Z_{a,b}}(z_1,z_3).
\ee

\noindent
{\em{Case II:}}
Assume that we have
$
d_{Z_{a,b}}(z_1,z_2)=
d_{taxi_{b}}(z_1,z_2)
$
and 
\be
d_{Z_{a,b}}(z_2,z_3)= (h-t_2)+(h-t_3)+d_a(x_2,x_3).
\ee
Then by $d_b \geq d_a$, triangle inequality in $\mathbb R$ and for  $d_a$, we get
\begin{align*}
D_{13}& =
d_{taxi_{b}}(z_1,z_2) +(h-t_2)+(h-t_3)+d_a(x_2,x_3) \\
& \geq 
d_a(x_1,x_2)+ |t_1-t_2|
+(h-t_2)+(h-t_3)+d_a(x_2,x_3)
\\
& \geq 
d_a(x_1,x_2)
+(h-t_1)+(h-t_3)+d_a(x_2,x_3)\\
&\geq 
d_a(x_1,x_3) +(h-t_1)+(h-t_3) \geq 
d_{Z_{a,b}}(z_1,z_3).
\end{align*}

\noindent
{\em{Case III:}}
Assume that 
\be
d_{Z_{a,b}}(z_1,z_2) =   (h-t_1)+(h-t_2)+d_a(x_1,x_2)
\ee
\be
d_{Z_{a,b}}(z_2,z_3)= (h-t_2)+(h-t_3)+d_a(x_2,x_3).
\ee
Then, since $h-t_2 \geq 0$ and triangle inequality for $d_a$, we get 
\begin{align}
D_{13}& =
 (h-t_1)+(h-t_2)+d_a(x_1,x_2) 
 \\
& \qquad + (h-t_2)+(h-t_3)+d_a(x_2,x_3) \\
& \geq 
(h-t_1)+d_a(x_1,x_2) 
+(h-t_3)+d_a(x_2,x_3)\\
& \geq 
d_a(x_1,x_3) +(h-t_1)+(h-t_3)\\
& \geq 
d_{Z_{a,b}}(z_1,z_3),
\end{align}
which completes the proof of the triangle inequality on $(Z_{a,b},d_{Z_{a,b}})$.

We have the claimed $1$-Lipschitz identity map, 
because $d_{Z_{a,b}}\le d_{taxi_{b}}$. 
We have the claims that $f_a$ and $f_b$ are distance preserving maps as follows.
By the definition of $d_{Z_{a,b}}$, $f_a$ and $d_b \geq d_a$,
$$
d_{Z_{a,b}}(f_a(x), f_a(y))=  d_{Z_{a,b}}
((x,h), (y,h))= \min\{d_b(x,y),d_a(x,y)\}=
d_a(x,y).
$$
By definition of $d_{Z_{a,b}}$, $f_b$,
\eqref{eq:ab-Zab} and $h=\epsilon/2$ we have
\be
d_{Z_{a,b}}(f_b(x), f_b(y))=
\min\{
d_b(x,y), 2h + d_a(x,y)\}
= d_b(x,y).
\ee
Furthermore, by the definition of $d_{Z_{a,b}}$, for any $x \in X$, it holds
\be 
d_{Z_{a,b}}((x,0),(x,h))=h, 
\ee
which implies that for any $R>h$,
\be
f_b(X)\subset T_R(f_a(X)) \textrm{ and }
f_a(X)\subset T_R(f_b(X)), 
\ee
so the Hausdorff distance
$d_H^Z(f_a(X), f_b(X)) \leq h$ which gives our final claim.
\end{proof}

\subsection{Proof of Theorem~\ref{thm:GH-cmpct}}

GH Convergence is now easy to prove:

\begin{proof}
First we apply Proposition~\ref{prop:unif-cmpct}
to get uniform convergence of $d_j$ to $d_\infty$
from below.
Then we apply Proposition~\ref{prop:Zab}
to $(X,d_j)$ and $(X, d_\infty)$ to see that
\be
d_{GH}((X,d_j),(X, d_\infty))<\epsilon_j \to 0.
\ee
\end{proof}

\begin{rmrk}\label{rmrk:not-GH}
Note that in the proof of Theorem~\ref{thm:GH-cmpct} it was essential that we assumed $(X,d_\infty)$ is compact, as seen
in Example~\ref{ex:not-unif}.   This
example has the monotonicity but no
Gromov-Hausdorff limit.  In fact, Gromov proved that if a sequence of
compact metric spaces has a GH limit then the entire sequence can be embedded via distance preserving maps
into a common metric space, $(Z,d_Z)$
\cite{Gromov-poly}.  However, the
Riemannian manifolds in Example~\ref{ex:not-unif} have an unbounded number of disjoint balls of radius $1$, so this is impossible.
\end{rmrk}

\subsection{A Common Compact Metric Space for the Whole Sequence}

Gromov proved in \cite{Gromov-poly} that an entire GH-converging sequence of metric spaces can be embedded with distance preserving maps into a common compact 
metric space $Z$ along with the limit so that the sequence of images converges in the Hausdorff sense inside $Z$ to the image of the limit. Here we construct a specific $Z$ which we will use later to prove SWIF convergence and Theorem~\ref{thm:Riem}. See Figure~\ref{fig:common-Z}.

\begin{prop}\label{prop:common-Z}
If $(X,d_j)$ is uniformly converging to
compact $(X, d_\infty)$ 
\be
d_\infty(x,y)-\epsilon_j<d_j(x,y)\le d_\infty(x,y) \quad \forall x,y\in X
\ee
with $\epsilon_j$ decreasing to $0$
then taking $Z_j$ of
Proposition~\ref{prop:Zab},
we define
\be
Z=\bigsqcup_{j=1}^\infty Z_j\,\, |_\sim
\textrm{ where }
Z_j=Z_{j,\infty}=X_j\times [0,h_j]
\ee
where $(x,0)\in Z_j$ is identified
with $(x,0)\in Z_k$ for all $x\in X$
and all $j,k \in {\mathbb N}$.  
Then $d_Z: Z \times Z \to [0, \infty)$ given by 
\be
d_Z(z_j,z_k)=  \inf\{d_{Z_j}(z_j,(x,0))+d_{Z_k}((x,0),z_k)\,|\, x\in X\} 
\ee
for 
\be
z_j=(x_j,t_j)\in Z_j
\textrm{ and }z_k=(x_k,t_k)\in Z_k \qquad j \neq k,
\ee
and, otherwise, 
\be
d_Z(z_j,z_k)=  d_{Z_j}(z_j,z_k),
\ee
so the inclusion
maps, 
\be
\zeta_j: (Z_j, d_{Z_j}) \to (Z, d_Z),
\ee 
are distance preserving.
In addition, taking
\be
f_j=f_{j,(j,\infty)}:(X, d_j) \to Z_j
\textrm{ and } f_{\infty,(j,\infty)}:(X,d_\infty)\to Z_j
\ee
from Proposition~\ref{prop:Zab}, we have 
distance preserving maps:
\be
\chi_j=\zeta_j\circ f_j: (X, d_j) \to 
 (Z, d_Z)
\ee
and
\be
\chi_\infty= \zeta_j\circ f_{\infty,(j,\infty)}: (X, d_\infty) \to 
 (Z, d_Z),
\ee
where $\chi_\infty(x)=\zeta_j(x,0)$ does not depend on $j$
and is an isometry onto,
\be
Z_0=\chi_\infty(X)=\zeta_j(f_{\infty,(j,\infty)}(X))\subset Z.
\ee
Furthermore, we have
\be \label{eq:dZchihj}
d_Z(\chi_j(x),\chi_\infty(x))=d_{Z_j}((x,h_j),(x,0))=h_j
\ee
and
\be \label{eq:dZchihj2}
d_H^Z(\chi_j(X)),\chi_\infty(X))\leq h_j.
\ee
Finally $(Z, d_Z)$ is compact.
\end{prop}

Intuitively we have glued the $Z_j$ together like the pages of a book, $\zeta_j:Z_j\to Z$, along the spine, $Z_0$, as in Figure~\ref{fig:common-Z}.

\begin{figure}[h] 
   \centering
   \includegraphics[width= \textwidth]{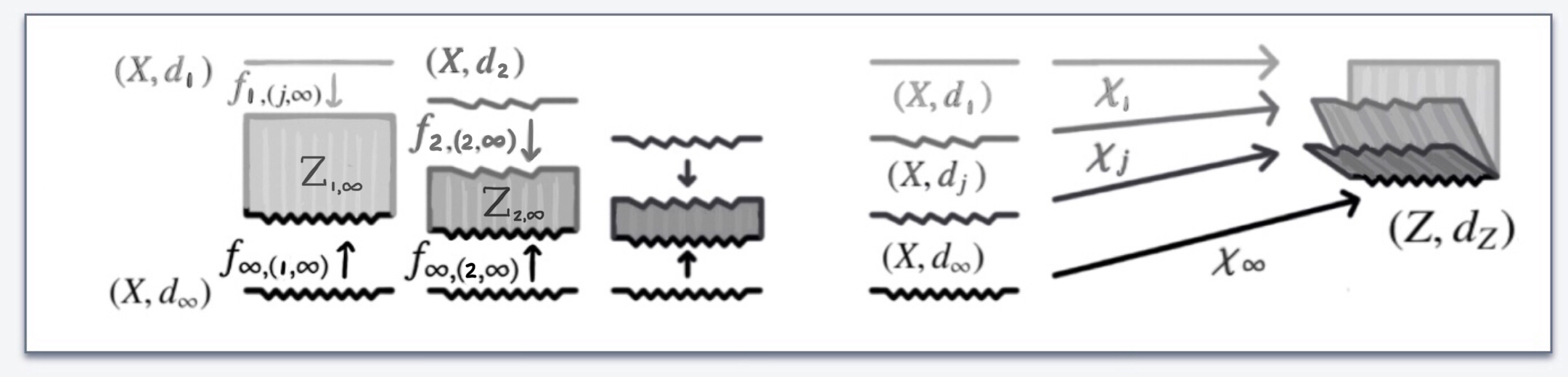} 
\caption{On the left, we see the sequence
$(X,d_j)$ above with downward maps $f_{j,(j,\infty)}:(X,d_j)\to
(Z_{j,\infty}, d_{Z_{j,\infty}})$ 
and upward maps $f_{\infty,(j,\infty)}:(X,d_\infty)\to
(Z_{j,\infty}, d_{Z_{j,\infty}})$ 
as constructed in Proposition~\ref{prop:Zab}
and depicted in Figure~\ref{fig:Zab}.
On the right, we glue together all the
$(Z_{j,\infty}, d_{Z_{j,\infty}})$ to create
$(Z,d_Z)$ as in Proposition~\ref{prop:common-Z}
with maps $\chi_j:X_j\to Z$ and 
$\chi_\infty:X_j\to Z$.
}
\label{fig:common-Z}
\end{figure}

\begin{proof}
It is clear that $d_Z$ is symmetric and positive definite. To see that it satisfies the triangle inequality we have several cases.
Let $z_{j_i}=(x_i,t_i) \in Z_{j_i}$, $j_i \in \mathbb N$ and $i=1,2,3$,   
\be
D_{13}=d_Z(z_{j_1},z_{j_2}) + d_Z(z_{j_2}, z_{j_3}).
\ee
{\bf{Case I:}} 
Assume that $j_1=j_2$ and $j_2=j_3$. 
In this case, $d_Z=d_{Z_{j_i}}$ which satisfies the triangle inequality since 
$d_{Z_{j_i}}$ is a distance function.

\noindent{\bf{Case II:}} 
Assume that $j_1=j_2$ and $j_2 \neq j_3$. 
By  $j_1=j_2$, the triangle inequality for $d_{Z_{j_2}}$ and definition of $d_Z$ we get
\begin{align}
    D_{13}= & d_{Z_{j_2}}(z_{j_1},z_{j_2}) + 
     \inf_{x \in X}\{d_{Z_{j_2}}(z_{j_2},(x,0))+ 
     d_{Z_{j_3}}((x,0),z_{j_3})\} \\
    = & 
 \inf_{x \in X}\{d_{Z_{j_2}}(z_{j_1},z_{j_2}) +  d_{Z_{j_2}}(z_{j_2},(x,0))+ 
     d_{Z_{j_3}}((x,0),z_{j_3})\}  \\
      \geq & 
     \inf_{x \in X}\{
     d_{Z_{j_2}}(z_{j_1},(x,0))+ 
     d_{Z_{j_3}}((x,0),z_{j_3})\} 
     =  d_{Z}(z_{j_1},z_{j_3}). 
    \end{align}
    
\noindent{\bf{Case III:}} 
Assume that $j_1 \neq j_2$ and $j_2 \neq j_3$. Take $x, x' \in X$, using the triangle inequality for $d_{Z_{j_2}}$ and $d_{Z_{j_1}}$, and the definition of $d_Z$ we get
\begin{align*}
   & d_{Z_{j_1}}(z_{j_1},(x,0)) +
    d_{Z_{j_2}}((x,0),z_{j_2})+
    d_{Z_{j_2}}(z_{j_2},(x',0)) +
d_{Z_{j_3}}((x',0),z_{j_3}) 
    \\  & \geq 
d_{Z_{j_1}}(z_{j_1},(x,0)) +
    d_{Z_{j_2}}((x,0),(x',0)) +
d_{Z_{j_3}}((x',0),z_{j_3}) 
\\  & \geq 
d_{Z_{j_1}}(z_{j_1},(x,0)) +
    d_{Z_{j_1}}((x,0),(x',0)) +
d_{Z_{j_3}}((x',0),z_{j_3}) 
\\  & \geq 
d_{Z_{j_1}}(z_{j_1},(x',0)) +
d_{Z_{j_3}}((x',0),z_{j_3}) \geq 
      d_{Z}(z_{j_1},z_{j_3}).
    \end{align*}
Taking the infimum over $x, x' \in X$ in the previous expression we obtain $D_{13} \geq d_{Z}(z_{j_1},z_{j_3})$. This concludes the proof that $(Z, d_Z)$ is a metric space. 

The inclusion
maps $\zeta_j$ are distance preserving by their definition and the definition of $d_Z$. Since by Proposition~\ref{prop:Zab}
$f_j=f_{j,(j,\infty)}$ and $f_{\infty,(j,\infty)}$ are distance preserving maps, then the compositions $\chi_j=\zeta_j\circ f_j$ and 
$\chi_\infty= \zeta_j\circ f_{\infty,(j,\infty)}$ are distance preserving maps. 
Since we are identifying points of the form $(x, 0) \in Z_j$ with points $(x, 0) \in Z_k$ for all $x \in X$, then $\chi_\infty(x)=\zeta_j(x,0)$ does not depend on $j$. Furthermore, $\chi_\infty$
is an isometry onto $Z_0=\chi_\infty(X)$.

Since $\chi_j$ and 
$\chi_\infty$ are distance preserving maps and by \eqref{eq-GHestimate} in Proposition \ref{prop:Zab} we have
\be
d_Z(\chi_j(x),\chi_\infty(x)) \leq d_{Z_j}((x,h_j),(x,0)) \leq h_j
\ee
so we have (\ref{eq:dZchihj})-(\ref{eq:dZchihj2}).

Finally we check that $(Z, d_Z)$ is compact. 
Given any sequence $z_i \in Z$,
\be
\exists \, j_i\in {\mathbb N} \,\,\exists\, (x_i,t_i)\in Z_{j_i}=X\times[0,h_{j_i}] \textrm{ such that }
z_i=\zeta_{j_i}(x_i,t_i).
\ee
Since $(X,d_\infty)$ is compact, a subsequence $x_i\to x_\infty$.   Since $[0,1]$ is compact, a further
subsequence can be taken such that $t_i\to t_\infty$.
A further subsequence can be taken to guarantee that
either $j_i\to \infty$ or $j_i$ is constant $j_i=j_0$.
In the diverging case we use $t_i\le h_{j_i}\to 0$ to see that
$t_\infty=0$, thus our subsequence converges to 
\be
z_\infty=\chi_\infty(x_\infty)=\zeta_{j_i}(x_\infty,0)\in Z_0
\ee
which we see as follows:
\begin{eqnarray}
d_Z(z_i,z_\infty) &=& d_Z(\zeta_{j_i}(x_i,t_i),\zeta_{j_i}(x_\infty,0))\\
&=&d_{Z_{j_i}}((x_i,t_i),(x_\infty,0))\\
&\le& d_{\infty}(x_i,x_\infty)+|t_i-0| \to 0.
\end{eqnarray}
In the constant, $j_i=j_0$ case, our subsequence converges to 
\be
z_\infty=\zeta_{j_0}(x_\infty,t_\infty)
\ee
which we see as follows:
\begin{eqnarray}
d_Z(z_i,z_\infty) &=& d_Z(\zeta_{j_0}(x_i,t_i),\zeta_{j_0}(x_\infty,t_\infty))\\
&=&d_{Z_{j_0}}((x_i,t_i),(x_\infty,t_\infty))\\
&\le& d_{\infty}(x_i,x_\infty)+|t_i-t_\infty| \to 0.
\end{eqnarray}
Thus $Z$ is compact.
\end{proof}

\begin{lem}\label{lem:W-to-X}
Assume the hypotheses of 
Proposition~\ref{prop:common-Z}.
For all $j\in {\mathbb N}$ there exists 
\be
W_j =\bigcup_{k=j}^\infty \chi_k(Z_k)\subset Z
\ee
and there is a $1$-Lipschitz map,
$
F_j: (W_j, d_Z)\to (X,d_j)$  where 
\be
F_j(\zeta_k(x,t))=x \quad \forall\, k \ge j\, \forall x\in X, \, \forall t\in [0, h_k].
\ee
\end{lem}

\begin{proof}
Let $z_1,z_2\in W_j\subset Z$.   Then 
\be
\exists k_i \geq j,\,\,\exists x_i\in X\,\,\exists t_i\in [0,h_{k_i}] 
\textrm{ s.t. }
z_i = \chi_{k_i}(x_i,t_i).
\ee
If $k_1=k_2$, since 
$d_\infty \geq d_{k_i}\geq d_j$, by the minimum in the definition of $d_{Z_{k_i}}$ and the monotonicity of the distance functions, we get
\be
d_{Z}(z_1, z_2)\ge d_{k_i}(x_1,x_2)\ge d_j(x_1,x_2).
\ee
If $k_2 \neq k_1$, then observe that for all $x\in X$,
\be
d_{Z_{k_i}}(z_i,(x,0))=
d_{Z_{k_i}}((x_i,t_i),(x,0))
\ge d_{k_i}(x_i,x)\ge d_{j}(x_i,x)
\ee
by the  definition of $d_{Z_{k_i}}$ 
because $t_i\ge 0$
and $h_{k_i}-t_i+h_{k_i} \ge 0$,
followed by
the monotonicity of the distance functions.
Thus
\be
d_{Z_{k_1}}(z_1,(x,0))+d_{Z_{k_2}}((x,0),z_2)
\ge d_{j}(x_1,x)+d_{j}(x_2,x). 
\ee
Taking the infimum over all $x\in X$ on the left and applying the
triangle inequality for $d_j$ on the right we have,
\be
d_Z(z_1,z_2)\ge d_{j}(x_1,x_2).
\ee
\end{proof}

\section{Intrinsic Flat Convergence}
\label{sect:SWIF}

In this section we review the definition of
integral current spaces which include all oriented Riemannian manifolds of finite volume with boundary of finite volume.  We then review
the Sormani-Wenger Intrinsic Flat (SWIF) distance between these spaces.
Finally we state and prove Theorem~\ref{thm:SWIF-cmpct} concerning the convergence of monotone increasing sequences of integral current spaces which implies the SWIF convergence claimed in Theorem~\ref{thm:Riem}.

\subsection{Lipschitz Functions and Tuples on Metric Spaces}

We say that a function, $F:(X,d_X)\to (Y,d_Y)$
is $K$-Lipschitz if there exists $K \geq 0$ such that
\be\label{eq:LipK}
d_Y(F(x_1,x_2))\le K d_X(x_1,x_2) \quad \forall x_1,x_2\in X.
\ee
We define the Lipschitz constant of $F$ to be
the smallest such $K$.

\begin{lem}
\label{lem:Lip}
Suppose that $(X,d_a)$ and $(X, d_b)$ are two metric spaces and 
\be \label{eq:ab1}
d_a(x_1,x_2) \le d_b(x_1,x_2)\quad \forall x_1,x_2\in X.
\ee
If a function $\pi:X\to Y$
is $K$-Lipschitz as a map $\pi: (X,d_a)\to (Y,d_Y)$
then it is also $K$-Lipschitz as a map 
$\pi: (X,d_b)\to (Y,d_Y)$.   
\end{lem}

\begin{proof}
\be
d_Y(\pi(x_1),\pi(x_2))
\le K d_a(x_1,x_2)\le K d_b(x_1,x_2)
\quad \forall x_1,x_2\in X.
\ee
\end{proof}

By the definition in \cite{AK}
an m-tuple, $(\pi_0,\pi_1,...,\pi_m)$,
of Lipschitz functions with respect to $d_a$ consists
of a bounded Lipschitz function,
$\pi_0: (X, d_a)\to {\mathbb R}$, and 
$m$ Lipschitz functions,
$\pi_k: (X, d_a)\to {\mathbb R}$, for $k=1,\dots, m$.

\begin{cor}\label{cor:tuple}
Suppose that $(X,d_a)$ and $(X, d_b)$ are two metric spaces and 
\be \label{eq:abcor}
d_a(x,y) \le d_b(x,y)\quad \forall x,y\in X.
\ee
Then tuples of Lipschitz functions with respect to $d_a$ are 
tuples of Lipschitz functions with respect to $d_{b}$.
\end{cor}

\begin{lem}\label{lem:Lip-Z}
Given the hypotheses of Proposition~\ref{prop:common-Z} and Lemma~\ref{lem:W-to-X}.
If $\pi:(X, d_j) \to Y$ is
Lipschitz $K_j$, and 
$F_j:(W_j,d_Z)\to (X, d_j)$ of Lemma~\ref{lem:W-to-X},
then
\be\label{eq:tilde}
\tilde{\pi}:(W_j, d_Z)\to Y
\textrm{ such that }
\tilde{\pi}(z)=\pi \circ F_j(z)
\ee
is also Lipschitz $K_j$ and satisfies
\be
\tilde{\pi}(\chi_\infty(x))=\pi(x)\quad \,\forall x\in X.
\ee
\end{lem}

\begin{rmrk}\label{rmrk:Lip-Z}
Note that if we try to define
$\tilde{\pi}$ as a function from all
of $(Z,d_Z)$ to ${\mathbb R}$
as in (\ref{eq:tilde}) then it might not
be Lipschitz.   The proof depends
strongly on the restriction to
$W_j\subset Z$ and Lemma~\ref{lem:W-to-X}.
\end{rmrk}

\begin{proof}
Recall
$F_j:(W_j,d_Z)\to (X, d_j)$
are Lipschitz one functions, so
\be
d_{Y}(\tilde{\pi}(z_1),
\tilde{\pi}(z_2))
\le K_j \,d_j(F_j(z_1),F_j(z_2))
\le K_j \,d_{Z}(z_1,z_2)
\quad \forall z_1,z_2\in Z.
\ee
\end{proof}

\subsection{Charts into our Metric Space}

We say that a map $\psi: U\subset {\mathbb R}^m \to X$ from a Borel subset $U$ in ${\mathbb R}^m$ into a metric space, $(X, d_X)$, is a Lipschitz chart if there exists $K_{\psi} > 0$
such that
\be\label{eq:LipchartK}
d_X(\psi(a_1),\psi(a_2))\le K_\psi |a_1-a_2| \quad \forall a_1,a_2\subset U.
\ee
These charts will be applied in the next section to define rectifiable and integral currents, and then
integral current spaces in the following section.

\begin{lem}
\label{lem:charts}
Suppose $(X,d_a)$ and $(X, d_b)$ are two metric spaces and 
\be \label{eq:abcharts}
d_a(x,y) \le d_b(x,y)\quad \forall x,y\in X.
\ee
If 
$\psi: U\subset {\mathbb R}^m \to X$ is a Lipschitz chart with respect to $d_b$ on $X$, then it is 
Lipschitz with respect to $d_{a}$
 on $X$.   
\end{lem}

\begin{proof}
We have a Lipschitz constant $K_\psi\in (0,\infty)$ such that
\be
d_b(\psi(a_1),\psi(a_2))
\le K_\psi |a_1-a_2| \qquad \forall a_1,a_2\in U.
\ee
By (\ref{eq:abcharts}), we have
\be
d_a(\psi(a_1),\psi(a_2))
\le K_{\psi} |a_1-a_2| \qquad \forall a_1,a_2\in U.
\ee
\end{proof}

\begin{rmrk}
Sometimes in the definition of a rectifiable space or current it is written that the charts $\psi_i:U_i \to X$ are bi-Lipschitz. However, once one has a countable collection of Lipschitz charts, one can replace them with a collection of
bi-Lipschitz charts by Lemma 4 in Kirchheim's \cite{Kirchheim-rect} (cf. Lemma 4.1 of \cite{AK}).
The bi-Lipschitz charts are not necessarily the same collection of charts.
If we apply our Lemma~\ref{lem:charts} to a collection of charts $\psi_i:(U_i\subset{\mathbb R}^m, d_{\mathbb R^m})\to (X,d_b)$ that are bi-Lipschitz onto their images, we conclude they are Lipschitz as charts into $(X,d_a)$ but cannot determine if they are bi-Lipschitz.  We would need to apply Kirchheim's theorem to produce a new collection of bi-Lipschitz charts into $(X,d_b)$ if we need them.
\end{rmrk}

\subsection{Ambrosio-Kirchheim's integral currents on complete metric spaces}

Here we review the work of
Ambrosio-Kirchheim in \cite{AK}.
They defined $m$ dimensional currents on complete metric spaces, $(Z,d_Z)$ as
multilinear functionals, $T$, acting on
Lipschitz tuples satisfying various hypotheses \cite{AK}.  For example, given a compact $m$-dimensional oriented Riemannian manifold, $(M,g)$, we can define the 
current,
\be
[[M]](\pi_0,...,\pi_m)
=\int_M \pi_0 \, d\pi_1\wedge \cdots\wedge d\pi_m
\ee
which is well defined because Lipschitz
maps are differentiable almost everywhere.  The integration is actually defined by integrating over atlas of disjoint oriented charts and taking the sum.

An $m$ dimensional {\em integer rectifiable current} on a 
complete metric space, $(Z,d_Z)$, is a multilinear functional, $T$, defined
on Lipschitz tuples on $(Z,d_Z)$ that has a
parametrization
which is a countable collection
of Lipschitz charts with non-zero Lipschitz constants
\be
\psi_i:(U_i\subset {\mathbb R}^m,d_{{\mathbb R}^m}) \to (Z,d_Z)
\ee
defined on Borel sets, $U_i\subset {\mathbb R}^m$,
 with integer multiplicities $a_i$,
 such that total Hausdorff measures satisfies
 \be \label{eq:Haus-finite}
 \sum_{i=1}^\infty |a_i| \mathcal{H}_{d_Z}^m(\psi_i(U_i))<\infty.
\ee
The integer
rectifiable current $T$ is defined by
the weighted sum:
\be \label{eq:defn-current-1}
T=\sum_{i=1}^\infty \, a_i\, \psi_{i\#}[[U_i]],
\ee
where
\be \label{eq:defn-current-2}
\psi_{\#}[[U]](\pi_0,\pi_1,...,\pi_m)
= \int_{U} (\pi_0\circ \psi)
d(\pi_1\circ \psi)\wedge\cdots
\wedge d(\pi_m\circ \psi).
\ee
This integral is well defined because the
charts are Lipschitz and thus
differentiable almost everywhere.
The weighted sum of the integrals
defining $T(\pi_0,...,\pi_m)$
is finite because the weighted sum of
the
Hausdorff measures in (\ref{eq:Haus-finite}), $\max\{\pi_0\}$,
and the product $\Lip(\pi_j)$ for
$j=1, \ldots, m$ can be used to bound the
integrals.  Warning: the Lipschitz charts and the Hausdorff measures depend on the distance function on $Z$.
See Lemma~\ref{lem:current} below.

Note that if $(M^m,g)$ is a compact oriented Riemannian manifold with boundary, with a smooth atlas of oriented charts $\psi_i:(W_i\subset {\mathbb R}^m,d_{{\mathbb R}^m}) \to (M^m,d_g)$ then we can choose $U_i\subset W_i$ such that
$\psi_i(U_i)$ are disjoint and their union covers $M^m$.
This defines a natural current structure for $M$
with weight $1$, which agrees with integration of
differential $m$ forms:
\be
[[M]](\pi_0,...,\pi_m)=\sum_{i=1}^N \, \psi_{i\#}[[U_i]]=\int_M \pi_0 \, d\pi_1\wedge \cdots \wedge d\pi_m.
\ee
By Stoke's Theorem, for any collection of
smooth functions, $\pi_i:M \to {\mathbb R}$,
for $i=1,...,m-1$, we have
\be
\int_{\partial M} \pi_1 \, d\pi_2\wedge 
\cdots \wedge d\pi_{m-1}
= \int_M 1 \, d\pi_1 \wedge\cdots \wedge d\pi_{m-1}.
\ee
To be consistent with Stoke's Theorem,
Ambrosio-Kirchheim define the boundary of a current, $T$,
as follows
\be \label{eq:defn-bndry}
\partial T(\pi_1,...,\pi_{m-1})
=T(1,\pi_1,...,\pi_{m-1}),
\ee
for any Lipschitz tuple $(\pi_1,...,\pi_{m-1})$ on $Z$.

They define an {\em integral current} to be an integer
rectifiable current whose boundary
is also integer rectifiable.   Be warned however that the parametrization of the boundary, $\partial T$, is not necessarily found by taking the boundaries of the atlas of charts parametrizing, $T$.

The {\em mass measure}, $||T||_{d_Z}$, 
of a current $T$ on a complete metric space, $(Z,d_Z)$, is the smallest finite Borel measure such that
\be
T(\pi_0,\pi_1,...,\pi_m)
\le \prod_{i=1}^m \Lip(\pi_i)\,\,\int_Z |\pi_0| \,\,||T||_{d_Z} 
\ee
for all tuples and the {\em mass} of $T$ is
\be\label{eq:defn-mass}
\mass_{d_Z}(T)=||T||_{d_Z}(Z).
\ee
The support of a current
is
\be\label{eq:defn-spt}
\spt_{d_Z}(T)=
\{z\in Z\,:\, ||T||_{d_Z}(B_{d_Z}(z,r))>0 \,\, \forall r>0\}
\ee
and the set of an
$m$ dimensional current
is
\be\label{eq:defn-set}
\set_{d_Z}(T)=
\left\{z\in Z\,:\, \liminf_{r\to 0} ||T||_{d_Z}(B_{d_Z}(z,r))/r^m>0 \,\, \forall r>0\right\}.
\ee
Ambrosio-Kirchheim prove that $\set_{d_Z}(T)$ is rectifiable and that
the closure of this
set is the support of $T$,
\be\label{eq:spt-set}
\spt(T)=\overline{\set_{d_Z}(T)}.
\ee

Given a Lipschitz map, $F:(X,d_X)\to(Y,d_Y)$,
the pushforward of an integral current, $T$, on
$(X,d_X)$, to an integral current, $F_\#(T)$,
on $(Y,d_Y)$ is defined by
\be\label{eq:defn-push}
F_\#(T)(\pi_0,...,\pi_m)=T(\pi_0\circ F,...,\pi_m\circ F)
\ee
for any Lipschitz tuple on $(Y,d_Y)$.  To see that
$F_\#(T)$ is an integral current, Ambrosio-Kirchheim
prove that
the mass is bounded as follows:
\be\label{eq-massPush}
\mass_{d_Y}(F_\#(T))\le (\Lip(F))^m \mass_{d_X}(T).
\ee
For our purposes we need the following lemma.

\begin{lem}\label{lem:current}
Suppose $(X,d_a)$ and $(X, d_b)$ are two metric spaces and 
\be \label{eq:abcurrent}
d_a(x,y) \le d_b(x,y)\quad \forall x,y\in X.
\ee
If $S$ is an integral current on $(X,d_b)$ then it is an integral current on $(X,d_a)$ defined using the same collection of Lipschitz charts and weights with
\be\label{eq:mass-S}
\mass_{d_a}(S)\le \mass_{d_b}(S)
\textrm{ and }
\mass_{d_a}(\partial S)\le \mass_{d_b}(\partial S).
\ee
The mass measures also satisfy
\be\label{eq:measure-S}
||S||_{d_a} \le ||S||_{d_b}
\textrm{ and }
||\partial S||_{d_a} \le ||\partial S||_{d_b}.
\ee
\end{lem}

\begin{proof}
By the definition of integral current structure on $(X,d_b)$, there is a countable collection of $d_b$-Lipschitz charts, $\psi_i:U_i \subset {\mathbb R}^m \to X$,
and integers $a_i$ such that
\be\label{eq:S-acts}
S(\omega)=
\sum_{i=1}^\infty a_i \psi_{i \#}[[U_i]](\omega)
\ee
for any $m$ tuple, $\omega$, of $d_b$-Lipschitz functions.
By Lemma~\ref{lem:charts} these charts
are also $d_a$-Lipschitz charts, so
the weighted sum in (\ref{eq:S-acts}) is
well defined acting on any $m$ tuple, $\omega$, of $d_a$-Lipschitz functions. Note that 
the same can be done for $\partial S$. 
Thus $S$ is an integral current on $(X,d_a)$.

 To estimate the inequalities of the mass measures, recall that
in \cite{AK} the mass measure of a current, $S$, on a metric space $(X,d)$ is defined as
 the smallest Borel measure, $\mu$, on $X$ such that
\be\label{eq:mass}
|S(\pi_0,\pi_1,...,\pi_m)|\le \prod_{k=1}^m \Lip_d(\pi_k)\int_X |\pi_0| \,d\mu 
\ee
for all tuples, $(\pi_0,\pi_1,...,\pi_m)$, of $d$-Lipschitz functions.
   By Corollary~\ref{cor:tuple}, for
any tuple, $(\pi_0,\pi_1,...,\pi_m)$, of $d_a$-Lipschitz functions, we also
have (\ref{eq:mass}) with $\mu=||S||_{d_b}$.  Thus the smallest
measure for the tuples of $d_a$-Lipschitz functions exists and is smaller,
\be
||S||_{d_a}\le ||S||_{d_b}.
\ee
The same holds for $\partial S$ which implies the total mass estimate in
 \eqref{eq:mass-S}.
\end{proof}

Note the above lemma implies that distance preserving maps push forward currents conserving their masses, boundaries, and boundary masses. 

\begin{rmrk}\label{rmrk:current}
The converse of Lemma~\ref{lem:current}
does not hold in general.  If $T$
is an integral current on $(X,d_a)$, it does not necessarily define an integral current on $(X,d_b)$
for $d_b\ge d_a$.
A parametrization of $T$ by Lipschitz charts into $(X,d_a)$ might not give
Lipschitz charts into $(X,d_b)$ because that
would require a converse to Lemma~\ref{lem:charts} (which we do not have).   In addition, 
$T$ does not necessarily act on tuples in $(X,d_b)$ because that
would require a converse to Lemma~\ref{lem:Lip} (which we do not have).  Note that if one could prove that either or both of the above were true for some special $T$, one would still need to verify that the mass of $T$ and $\partial T$ are finite with respect to $d_b$
to conclude that $T$ is an integral current on $(X,d_b)$.
\end{rmrk}

Soon we will be considering a sequence of metric spaces $(X,d_j)$ with an integral current $T$, and 
the limit, $(X,d_\infty)$. Since
$d_\infty\ge d_j$, we will need to overcome the lack of a converse to Lemma~\ref{lem:current}. We will see later in the proof of Lemma~\ref{lem:ex}, that we can sometimes try
to rewrite a parametrization of an integral
current on $(X,d_j)$ to obtain charts that
are also Lipschitz into $(X,d_\infty)$.  We will see later in the proof of 
Theorem~\ref{thm:SWIF-cmpct}, which is needed to
complete the proof of Theorem~\ref{thm:Riem},
that we can sometimes find
an integral current, $T_\infty$, on $(X,d_\infty)$ and later show $T_\infty=T$
in the following sense:

\begin{defn}\label{defn:equal-current}
If we have an integral current $T$
on $(X,d_a)$ and there is also
an integral current, $S$, on $(X,d_b)$ 
where $d_b\ge d_a$ such that  
\be
S(\pi_0, \pi_1,...,\pi_m)
=T(\pi_0, \pi_1,...,\pi_m)
\ee
for any tuple of $d_a$ Lipschitz functions on $X$, then we say that ``$S=T$
viewed
as an integral current on $(X,d_a)$''.
That is, we can use any
parametrization of $S$ and $\partial S $ with
charts into $(X,d_b)$ as a parametrization for $T$ and $\partial T$ with the same weights and
charts into $(X,d_a)$
respectively.  
\footnote{Note that Lemma~\ref{lem:current} applies to $S$
but not $T$ as discussed in Remark~\ref{rmrk:current}.}
\end{defn}

\subsection{Weak Convergence of Currents}
A sequence $T_{j}$
{\em converges weakly as currents} in $(Z,d_Z)$
to an integral current, $T_\infty$
iff
\be \label{eq:defn-wk}
T_{j}(\pi_0, \pi_1,...,\pi_m)
\to T_{\infty}(\pi_0, \pi_1,...,\pi_m)
\ee
for all tuples, $(\pi_0, \pi_1,...,\pi_m)$
on $(Z,d_Z)$.

We will apply the following powerful theorem of Ambrosio-Kirchheim:

\begin{thm}[Ambrosio-Kirchheim]\label{thm-AKcompact}
If $Z$ is compact and $T_j$ have
uniform upper bounds
\be
\mass(T_j)\le V_0 \textrm{ and }\mass(\partial T_j) \le A_0
\ee
then a subsequence $T_{j_k}$ converges
weakly to an integral current, $T_\infty$,
and the mass measure is lower semicontinuous:
\be\label{eq:semi}
||T_\infty||(U)\le \liminf_{k\to \infty} ||T_{j_k}||(U)
\ee
for any open set $U \subset Z$. 
\end{thm}

\subsection{Integral Current Spaces}

In \cite{SW-JDG}, Sormani and Wenger define an $m$-dimensional {\em integral current space}, $(X,d,T)$, as a metric space, $(X,d)$, with an integral current, $T$, defined on the closure of $(X,d)$
such that the
$\set(T)=X$ where
\be\label{eq:defn-set-2}
\set(T)=\left\{ x\in \bar{X} \,:\,
\liminf_{r\to 0} ||T||_d(B_d(x,r))/r^m >0\right\}.
\ee
Such metric spaces $(X, d)$ are rectifiable and parametrized by the charts that parametrize their integral current structure, $T$.

Given a smooth compact oriented $m$-dimensional Riemannian manifold with boundary, $(M,g)$, it can be viewed as an integral current space, $(M, d_g,[[M]])$.
Here the mass measure is the volume measure,
\be \label{eq:mass=vol}
\mass_{d_g}(M)=\vol_g(M),
\ee
and
\be \label{eq:set=M}
\set_{d_g}([[M]])=
\left\{ x\in M\,:\,
\liminf_{r\to 0} \vol_g(B_{d_g}(x,r))/r^m >0\right\}=M.
\ee

The following lemma depicted in Figure~\ref{fig:cusp-cone} is useful for
constructing examples:

\begin{lem}\label{lem:ex}
Suppose the sphere
${\mathbb S}^2$ is endowed with
the standard round
metric, $g_{{\mathbb S}^2}=dr^2+\sin^2(r) \,d\theta^2$, 
defined 
on a sphere parametrized by $r\in [0,\pi]$
and $\theta\in [0,2\pi]$.
We define the possibly singular metric, $g_h$,
induced by a graph, $h\circ r:{\mathbb S}^2 \to [0,\infty)$
so that
\be
g_h=dr^2+\sin^2(r) \,d\theta^2+(h'(r))^2 dr^2
\ge
g_{{\mathbb S}^2}=dr^2+\sin^2(r)\, d\theta^2
\ee
where $h(r(x))$ is smooth on $r^{-1}(0,\pi)$, increasing on
$[0,\pi/2]$ and decreasing on
$[\pi/2,\pi]$ with $h(r)=h(\pi-r)$,
but possibly only continuous at the poles
$p_0,p_\pi$ where $r=0,\pi$
respectively.   We can define
a unique compact metric space,
$({\mathbb S}^2,d_h)$,
using $g_h$-lengths of curves, $C$,
such that 
\be \label{eq:d_h}
C:(0,1)\to {\mathbb S}^2\setminus\{p_0,p_\pi\}
\ee
with 
bounded radial arclength functions from the poles $p=p_0,p_\pi$, 
\be
s_p(r(x))=d_h(x,p)=
\int_0^{r(x)}\sqrt{1+h'(r)^2}\, dr
\textrm{ with } \lim_{r\to 0}s_p(r)=0.
\ee
So that 
$({\mathbb S}^2, d_{h})$ is the metric
completion of 
$({\mathbb S}^2\setminus \{p_0,p_\pi\}, d_{g_h})$.

We can find a parametrization of the standard integral current $[[{\mathbb S}^2]]$ on $({\mathbb S}^2, d_{{\mathbb S}^2})$
by a countable collection of Lipschitz charts into
$({\mathbb S}^2,d_h)$ which map
onto $r^{-1}(0,\pi)$ so that
we have an integral current space
\be\label{lem:ex-set}
\left(\set_{d_h}([[{\mathbb S}^2]]), d_h, T_h\right)
\textrm{ where }
r^{-1}(0,\pi)\subset \set_{d_h}(T_h)\subset{\mathbb S}^2,
\ee
and where ``$T_h=[[{\mathbb S}^2]]$
viewed
as an integral current on $({\mathbb S}^2,d_{{\mathbb S}^2})$''
in the sense of Definition~\ref{defn:equal-current}.   

So $\spt_{d_h}(T_h)={\mathbb S}^2$
and $\partial T_h=0$.
The poles are included in 
$\set_{d_h}(T_h)$
iff
\be \label{eq:h'0}
\lim_{r\to 0} |h'(r)|<\infty.
\ee
\end{lem}

\begin{figure}[h] 
   \centering
   \includegraphics[width=.8 \textwidth]{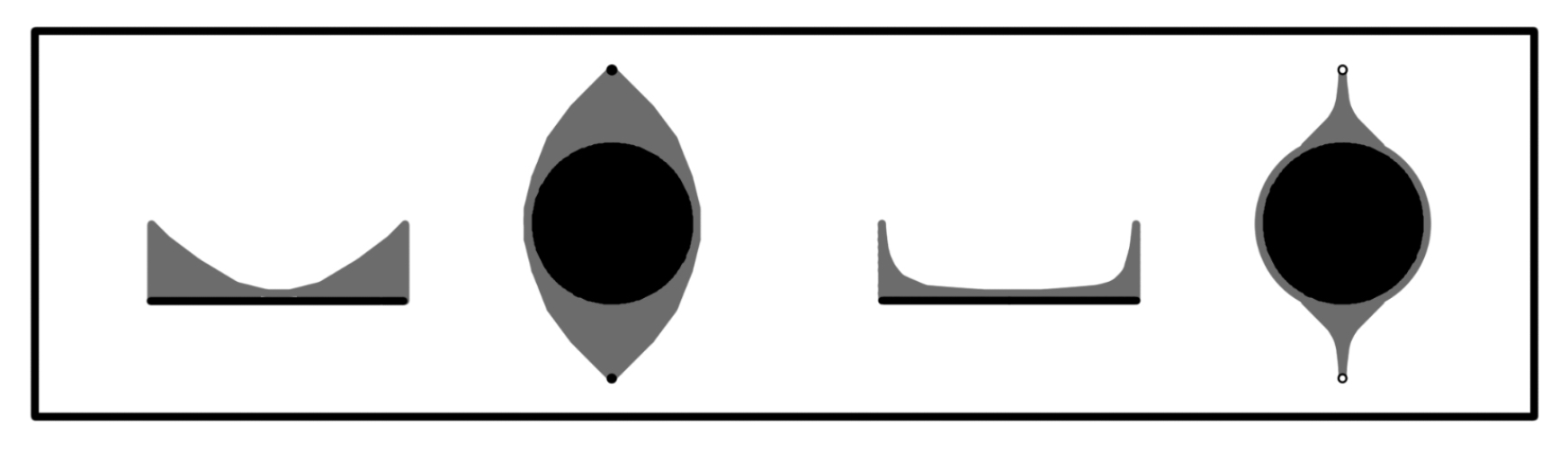} 
\caption{Lemma~\ref{lem:ex} applied 
with $h=h_{cone}$ of 
Example~\ref{ex:cone} on the left (which will include the singular poles) and
with $h=h_{cusp}$ of Example~\ref{ex:cusp}
on the right (which will not include the singular poles due to the cusp).
   }
   \label{fig:cusp-cone}
\end{figure}

\begin{proof}
We easily see that we can define
$d_h$ away from the poles using
lengths of curves that avoid the poles.
Near $p_0$, we have radial arclengths,
\be
s_{p_0}(x)=\int_0^{r(x)} \sqrt{1+h'(r)^2}\, dr
\le \int_0^{r(x)} 1+|h'(r)| dr<\infty
\ee
by the properties of $h$.
So
\be
\lim_{r\to 0}s_{p_0}(r)=
\lim_{r\to 0} r + |h(r)-h(0)| =0
\ee
and similarly for $p_\pi$.
This allows us to define $({\mathbb S}^2, d_h)$ 
as in (\ref{eq:d_h})
so that the identity map to
the standard round sphere is
a homeomorphism
and so
$({\mathbb S}^2, d_{{\mathbb S}^2})$
is compact with bounded
diameter.  

The
usual integral current,
$[[{\mathbb S}^2]]$, on the standard round sphere 
$({\mathbb S}^2, d_{{\mathbb S}^2})$ can be defined
using the standard single 
$(r,\theta)$ chart,
\be
\psi:[0,\pi]\times[0,2\pi]\to {\mathbb S}^2,
\ee
with multiplicity one, but this may not be a Lipschitz chart into
$({\mathbb S}^2, d_h)$
due to the singularities at the poles.   

Instead
we define a new parametrization
with a countable collection of
multiplicity one charts
avoiding the poles,
\be
\psi_i: U_i\to {\mathbb S}^2
\textrm{ where }\psi_i(r,\theta)=\psi(r,\theta),
\ee
and where $U_i$ are
defined inductively by
\be
U_1=[\delta_1,\pi-\delta_1]\times [0,2\pi)
\textrm{ and }
\, U_{i+1}=
[\delta_{i+1},\pi-\delta_{i+1}]
\times [0,2\pi)\setminus U_i
\ee
where $\delta_i$ decrease to $0$
so that these are Lipschitz charts
in $({\mathbb S}^2, d_h)$.  Then
\be \label{eq:T_h}
T_h=\sum_{i=1}^\infty \psi_{i\#}[[U_i]]
\ee
is a rectifiable current on $({\mathbb S}^2, d_h)$ because
\be
\sum_{i=1}^\infty {\mathcal{H}}^2(\psi_i(U_i))
\le \vol_{g_h}({\mathbb S}^2)<\infty.
\ee

We see that ``$T_h=[[{\mathbb S}^2]]$ 
viewed
as an integral current on $({\mathbb S}^2,d_{{\mathbb S}^2})$''
in the
sense of Definition~\ref{defn:equal-current}
because
\be
T_h(\pi_0,\pi_1,\pi_2)=[[{\mathbb S}^2]](\pi_0,\pi_1,\pi_2)
\ee
for any tuple in 
$({\mathbb S}^2, d_{{\mathbb S}^2})$,
since missing two points does not affect the integration of tuples.   
So $(\set_{d_h}(T_h), d_h,
T_h)$, is an integral current space. 
Furthermore
\be
\partial T_h = \partial [[{\mathbb S}^2]]=0.
\ee

In order to determine which points lie in $\set_{d_h}(T_h)$ we use (\ref{eq:defn-set})
and the fact that the mass measure 
is the volume measure as
stated above
(\ref{eq:mass=vol}).   It is easy to see that points where $g_h$ is smooth lie in the set, so we
conclude (\ref{lem:ex-set}) holds.
Taking either pole
$p\in \{p_0,p_\pi\}$ and 
$r=r_p$, and
applying l'Hopital's rule
twice, we have
\begin{eqnarray}
\lim_{R\to 0}\tfrac{1}{R^2}\vol_{g_h}(B_{d_h}(p,R))
&=&
\lim_{R\to 0}\tfrac{1}{R^2}\int_0^R 2\pi r(s) \, ds
=
\lim_{s\to 0}\frac{2\pi r(s)}{2R}\\ 
&=&\lim_{s\to 0}\pi r'(s)
=\lim_{s\to 0}\pi (1+(h'(r))^2)^{-1/2}.
\end{eqnarray}
So $p\in \set_{d_0}(T_h)$
iff this limit is $>0$
which happens iff (\ref{eq:h'0})
holds.
\end{proof}

\begin{ex}\label{ex:cone}
If we consider $g_h$
of Lemma~\ref{lem:ex}
defined using
\be \label{h-cone}
h(r)=h_{cone}(r)=1-{|\sin(r)|}
\ee
then $({\mathbb S}^2, d_{cone})$
is a compact
metric space with conical singularities, and
$(\set_{d_{cone}}([[{\mathbb S}^2]], d_h,
[[{\mathbb S}^2]])$ is an integral current space
with
\be
\set([[{\mathbb S}^2]])=
{\mathbb S}^2.
\ee
as depicted in Figure~\ref{fig:cusp-cone}
because
\be
\lim_{r\to 0} |h'(r)|=
\lim_{r\to 0} \cos(r)= 1<\infty.
\ee
\end{ex}

\begin{ex}\label{ex:cusp}
If we consider $g_{cusp}=g_h$
of Lemma~\ref{lem:ex}
defined using
\be \label{eq:h-cusp}
h(r)=h_{cusp}(r)=
1-\sqrt{\sin(r)}
=1-(\sin^2(r))^{1/4}
\ee
then $({\mathbb S}^2, d_{cusp})$
is a compact
metric space with cusp singularities at the poles, 
and the triple
$(\set_{d_{cusp}}([[{\mathbb S}^2]]), d_h,
[[{\mathbb S}^2]])$ is an integral current space
with
\be
\set_{d_{cusp}}([[{\mathbb S}^2]])={\mathbb S}^2\setminus \{p_0,p_\pi\}.
\ee
as depicted in Figure~\ref{fig:cusp-cone}
because
\be
\lim_{r\to 0} |h'(r)|=
\lim_{r\to 0} 
\frac{1}{2}\frac{\cos(r)}{\sqrt{\sin(r)}}=\infty.
\ee
\end{ex}

Theorem~\ref{thm:Riem} is careful with the description
of the SWIF limit because the
space, $(M, d_{\infty}, [[M]])$,
might not be an integral current space, as it might fail to have
$\set([[M]])=M$ as in Example~\ref{ex:lim-cusp}
depicted in Figure~\ref{fig:lim-cusp}.

\begin{ex}\label{ex:lim-cusp}
Let us consider 
the sphere, ${\mathbb S}^2$
with a sequence of
metric tensors,
$g_j=g_{h_j}$
as in Lemma~\ref{lem:ex} 
defined using
\be \label{eq:h-seq}
h_j(r)=1+\tfrac{1}{j}-
\left(\tfrac{1}{j^4}+ \sin^2(r)\right)^{1/4}.
\ee
Note that 
$
h_j\circ r: {\mathbb S}^2 \to [0,1]
$
are smooth functions because
$sin^2(r)$ is a smooth nonnegative function on
a sphere even at the poles where $r=0,\pi$ and $u^{1/4}$
is smooth 
for $u>0$.  In fact 
\be \label{eq:h-diff}
|h_j'(r)|=\tfrac{1}{4}\left(\tfrac{1}{j^4}+ \sin^2(r)\right)^{-3/4}\left(2\sin(r)\cos(r)\right).
\ee
so
\be\label{eq:nabla-inc}
|\nabla(h_j\circ r)|=|h_j'(r)|
\le |h_{j+1}'(r)| =|\nabla(h_{j+1})|\le |h_{cusp}'(r)|
\ee
and $\nabla (h_j\circ r)$ has the same direction as $\nabla r$
for all $j\in \mathbb N$.
Thus we have a monotone increasing sequence of metric tensors as
described in Lemma~\ref{lem:Riem-mono}.  

The radial arclengths,
\be
s_j(x)=\int_0^{r(x)} \sqrt{1+h_j'(r)^2}\, dr
\le \int_0^{r(x)} 1+|h_{cusp}'(r)| dr,
\ee
are uniformly bounded,
so we have a uniform upper bound
on diameter.  

Observe that $h_j \to h_{cusp}$
of Example~\ref{ex:cusp}
smoothly away from the poles.
So $d_j\to d_{cusp}$ pointwise
away from the poles.  Since
the arclength parameters
also converge,
we see that $d_j\to d_{cusp}$
on ${\mathbb S}^2\times{\mathbb S}^2$.   Since $({\mathbb S}^2,d_{{\mathbb S}^2})$ is compact we have all the hypotheses
of Theorem~\ref{thm:Riem}.  This
is an example where the SWIF limit
is a proper subset of the GH limit:  
\be
\set_{d_{cusp}}\left([[{\mathbb S}^2]]\right)\neq {\mathbb S}^2.
\ee
\end{ex}

\begin{figure}[h] 
   \centering
   \includegraphics[width=.8\textwidth]{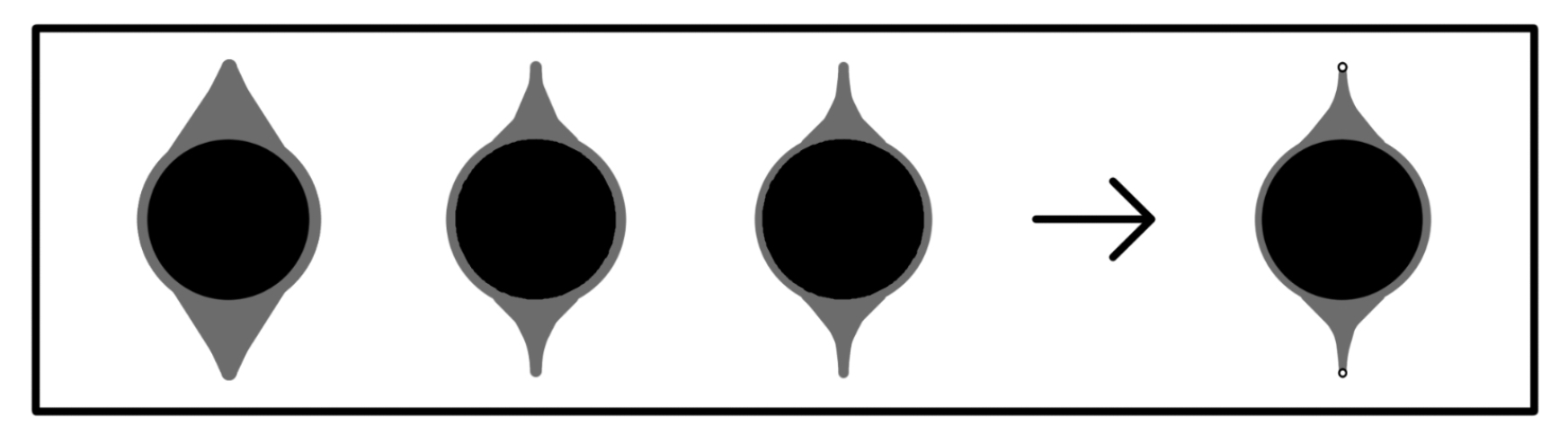} 
\caption{Example~\ref{fig:lim-cusp}
has $d_j\le d_{j+1}$ because
$|h_j'|\le |h_{j+1}'|$ and converges
to Example~\ref{ex:cusp} which has the poles
with cusp singularities removed.}
  \label{fig:lim-cusp} 
\end{figure}

\subsection{Intrinsic Flat Distance}

In \cite{SW-JDG}, Sormani and Wenger define the intrinsic flat distance between two integral current spaces as follows:

\begin{defn}\label{defn:SWIF}
\be
d_{SWIF}((X_a,d_a,S_a),
(X_b,d_b,S_b))
=
\inf\{ d_F^Z\left(f_{a\#}(S_a), f_{b\#}(S_b)\right)\}
\ee
where the infimum is taken over all complete metric spaces, $(Z,d_Z)$, and over all distance preserving maps,
\be
f_a: (X_a,d_a)\to (Z,d_Z)
\textrm{ and } f_b: (X_b,d_b)\to (Z,d_Z).
\ee
Here $d_F^Z$ denotes
the Flat distance
between integral currents
$T_a,T_b$ on $(Z,d_Z)$:
\be
d_F^Z\left(T_a, T_b \right)
=\inf \{\mass(A)+\mass(B):
A+\partial B=T_a-T_b\}
\ee
where the infimum is taken over all integral currents $A,B$ on $(Z,d_Z)$.
\end{defn} 

We say that a sequence
of integral current spaces
converges in the intrinsic flat sense, 
\be\label{eq:defn-SWIFto}
(X_j,d_j,S_j)\SWIFto
(X_\infty,d_\infty,S_\infty),
\ee
if 
\be
d_{SWIF}((X_j,d_j,S_j),
(X_\infty,d_\infty,S_\infty))\to 0.
\ee
In \cite{SW-JDG}, Sormani-Wenger prove this implies there
is a common complete metric space, $(Z,d_Z)$,
and distance preserving maps, $f_j:X_j\to Z$ such that, $T_j=f_{j\#}S_j$
converge weakly to 
``$T_\infty=f_{\infty\#}S_\infty$ as integral currents on $(Z,d_Z)$".  Thus by Ambrosio-Kirchheim semicontinuity
of mass,
\be
\liminf_{j\to \infty}
\mass_{d_j}(S_j) =
\liminf_{j\to \infty}
\mass_{d_Z}(T_j) \ge \mass_{d_Z}(T_\infty)=\mass_{d_\infty}(S_\infty).
\ee

We say the sequence 
of integral current spaces converges in the {\em volume preserving intrinsic flat
sense} if we have
(\ref{eq:defn-SWIFto})
and
\be
\mass_{d_j}(S_j)\to
\mass_{d_\infty}(S_\infty).
\ee
This notion was first studied by Portegies in \cite{Portegies-evalue}.   Additional work was
completed by Jauregui-Lee in \cite{Jauregui-Lee-VF}.
The consequences of $\mathcal{VF}$ convergence
are most recently reviewed by 
Jauregui-Perales-Portegies in \cite{Jauregui-Perales-Portegies}.

\subsection{Intrinsic Flat Convergence Theorem}

We can now state our SWIF convergence theorem:

\begin{thm}\label{thm:SWIF-cmpct} 
Suppose $(X,d_j)$ is a monotone increasing sequence 
of metric spaces satisfying all the hypotheses of Theorem~\ref{thm:GH-cmpct} so that $d_j\to d_\infty$ uniformly and 
\be
(X,d_j)\GHto (X,d_\infty)
\textrm{ which is compact}.
\ee
If there is a current, $T$, 
which is an integral current
structure for $(X,d_j)$,
\be \label{eq:SWIF-setj-T}
\set_{d_j}(T)=X \quad \forall j\in \mathbb{N},
\ee
with uniform upper bounds on total mass,
\be \label{eq:mass-est}
\mass_{d_j}(T)\le V_0
\textrm{ and } \mass_{d_j}(\partial T)\le A_0 \quad \forall j\in \mathbb{N},
\ee
then 
\be \label{eq:SWIF-T}
(X,d_j, T) \Fto (M_\infty,d_\infty,T_\infty).
\ee
where ``$T_\infty=T$ viewed
as an integral current on $(X,d_j)$''
in the sense
defined in Definition~\ref{defn:equal-current}
and
\be \label{eq:SWIF-set-T}
M_\infty=\set_{d_\infty}(T_\infty)\subset X \textrm{ with }
\overline{M}_\infty=X.
\ee
Furthermore,
\be
\mass_{d_\infty}(T_\infty)=
\lim_{j\to\infty}\mass_{d_{j}}(T).
\ee
so we have volume preserving intrinsic flat (${\mathcal{VF}}$)
convergence.
\end{thm}

Recall that in Example~\ref{ex:lim-cusp} we saw that the SWIF limit, $M_\infty$,
might be a proper subset of the
Gromov-Hausdorff limit, $X$.

\subsection{The Proof of Riemannian Theorem~\ref{thm:Riem}}
\label{sect:Riem}

Before we prove  Theorem~\ref{thm:SWIF-cmpct}, we show in this subsection how it can be applied to
prove Theorem~\ref{thm:Riem}:

\begin{proof}
Note that the hypotheses of Theorem~\ref{thm:Riem} imply the hypotheses of
Theorem~\ref{thm:GH-cmpct}
and Theorem~\ref{thm:SWIF-cmpct} as follows:
\begin{itemize}
\item
Monotonicity of metric tensors in (\ref{eq:mono-g}) implies
monotonicity of distances as in Definition~\ref{defn:mono}
by Lemma~\ref{lem:Riem-ineq}.
\item
Adding the diameter bound in (\ref{eq:diam}) implies the
pointwise convergence of $d_j\to d_\infty$ in Lemma~\ref{lem:diam} and the
diameter hypothesis of Theorem~\ref{thm:GH-cmpct}.
\item
The assumption that $(M, d_\infty)$ is compact in Theorem~\ref{thm:Riem} 
implies the compactness assumption in Theorem~\ref{thm:GH-cmpct},
so that 
we have uniform convergence as in Proposition~\ref{prop:unif-cmpct} and a common
metric space $Z_j$ for $(M,d_j)$ and $(M,d_\infty)$ as in 
Proposition~\ref{prop:common-Z}
and 
$(M, d_j) \GHto (M, d_\infty)$
as in Theorem~\ref{thm:GH-cmpct}.
\item
The hypothesis that $(M,g_j)$ are smooth
oriented Riemannian manifolds
in Theorem~\ref{thm:Riem}
implies $(M,d_{g_j},[[M]]])$ are
integral current spaces
with $T=[[M]]$ by (\ref{eq:set=M}).
\item
The volume bounds in (\ref{eq:vol}) imply the mass
bounds in (\ref{eq:mass-est}) by (\ref{eq:mass=vol}).
\end{itemize}
Thus we may apply the Theorem~\ref{thm:SWIF-cmpct} and (\ref{eq:mass=vol})
to conclude the
volume preserving intrinsic
flat convergence in
(\ref{eq:VF-Riem})
to a limit space satisfying
(\ref{eq:set-Riem}).
\end{proof}

\subsection{Finding the Current Structure of the SWIF Limit for Theorem~\ref{thm:SWIF-cmpct}}

 In \cite{SW-JDG}, Sormani-Wenger 
 proved that when one has a sequence
 of integral current spaces which converge in the Gromov-Hausdorff sense and have uniformly bounded
 total mass as in (\ref{eq:mass-est}), then SWIF limit of the integral current spaces is a subset of the Gromov-Hausdorff limit.
 For Theorem~\ref{thm:SWIF-cmpct} we wish to show the closure
 of the SWIF limit is the GH limit, so we need stronger control on the SWIF limit's current structure.
To achieve this,
 we will repeat the steps of the
 argument in \cite{SW-JDG} using the special properties of the common metric space, $(Z,d_Z)$,
 that we constructed in
 Proposition~\ref{prop:common-Z}.   See Figure~\ref{fig:zeta}.

 \begin{figure}[h] 
   \centering \quad \includegraphics[width=\textwidth]{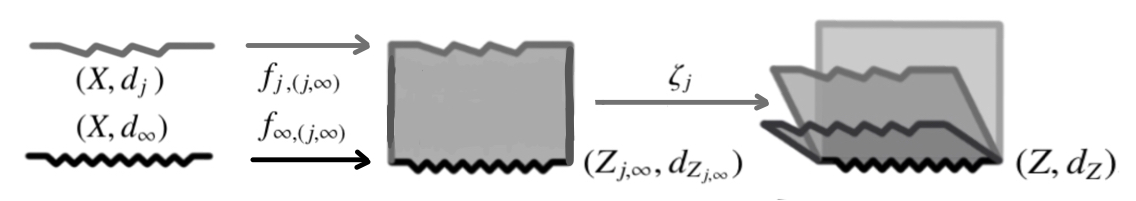} 
\caption{
The distance preserving maps
from Proposition~\ref{prop:common-Z}.
}\label{fig:zeta}
\end{figure}

\begin{prop}\label{prop:Tinfty}
Under the hypotheses of Theorem~\ref{thm:SWIF-cmpct},
we have the metric space
$(Z,d_Z)$ as constructed 
in Proposition~\ref{prop:common-Z} 
depicted in Figure~\ref{fig:zeta}
with distance preserving maps,
\begin{eqnarray}\label{eq:chi-to-push}
\chi_j\,&=&\zeta_j\circ f_{j\, ,(j,\infty)}: \,
(X,d_j\,)\to 
\,(Z_{j,\infty}, d_{Z_{j,\infty}}) \to 
(Z,d_Z),\\
\chi_\infty&=&\zeta_j\circ f_{\infty,(j,\infty)}: 
(X,d_\infty)\to (Z_{j,\infty},d_{Z_{j,\infty}})\to (Z,d_Z),
\end{eqnarray}
such that the pushforwards
\be\label{eq:push-Tj}
T_j=\chi_{j\#}(T)=\zeta_{j\#}(f_{j\#}(T))
\ee
are integral currents in $Z$ whose
supports are
\be\label{eq:spt-Tj}
\spt(T_j)=\zeta_j(f_j(X))= \zeta_j(X\times\{h_j\})
\subset \zeta_j(Z_{j,\infty})\subset Z
\ee
and have
\be\label{eq:mass-Tj}
\mass_{d_Z}(T_j)=\mass_{d_j}(T)\le V_0
\textrm{ and }
\mass_{d_Z}(\partial T_j)=\mass_{d_j}(\partial T)\le A_0.
\ee
Furthermore, 
there is a subsequence $T_{j_k}$
which converges weakly as currents in $(Z,d_Z)$
to an integral current, $\hat{T}_\infty$,
whose support lies on
the central spine in $Z$,
\be
\spt(\hat{T}_\infty)=Z_0=\chi_\infty(X)
=\zeta_j(X\times\{0\})
\subset \zeta_j(Z_{j,\infty})\subset Z.
\ee
\textcolor{black}{Taking $T_\infty=\chi_{\infty\#}^{-1}\hat{T}_\infty$ we have
``$T_\infty=T$
viewed
as an integral current on $(X,d_j)$''}
as in Definition~\ref{defn:equal-current},
and
\be \label{eq:mass-pres}
\mass_{d_\infty}(\textcolor{black}{T_\infty})=
\lim_{k\to\infty}\mass_{d_{j_k}}(T).
\ee
\end{prop}

In the next section we will apply this proposition to prove SWIF convergence holds without
needing a subsequence.   
 
\begin{proof}
Since the maps in (\ref{eq:chi-to-push})
are distance preserving,
the pushforward integral currents $T_j$ on $(Z,d_Z)$ defined as in (\ref{eq:push-Tj}),  satisfy
(\ref{eq:mass-Tj}) and (\ref{eq:spt-Tj}) following the definitions of
pushforward and mass (cf. \cite{SW-JDG}).

Since $(Z,d_Z)$ is compact by Proposition ~\ref{prop:common-Z}, we can apply the Ambrosio-Kirchheim Compactness Theorem, Theorem \ref{thm-AKcompact}, to
conclude that 
there is a subsequence $T_{j_k}$
which
converges weakly as currents in $(Z,d_Z)$
to an integral current, $\textcolor{black}{\hat{T}_\infty}$.

\vspace{-.2cm} 
 \begin{figure}[h] 
   \centering \quad \includegraphics[width=.6\textwidth]{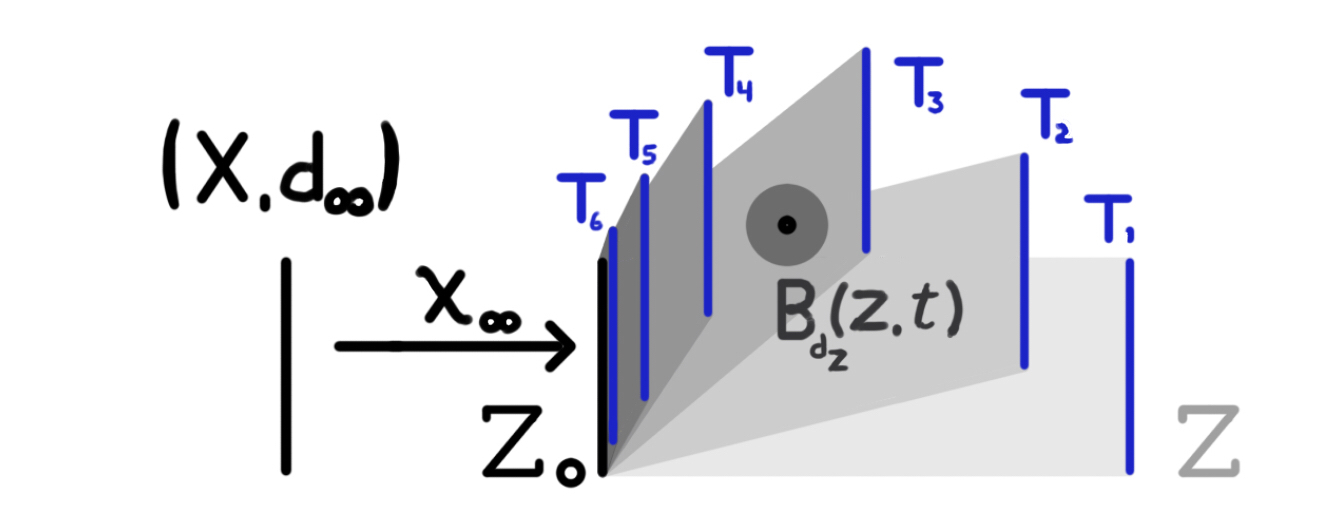} 
\caption{Here we see integral currents $T_j$ in $Z$ and $B_{d_Z}(z,t)$ of the proof that $\spt(\textcolor{black}{\hat{T}_\infty})\subset Z_0=\chi_\infty(X)\subset Z$ as claimed in (\ref{eq:spt-claim}).}
\label{fig:SWIF-Tj}
\end{figure}

First we claim 
\be \label{eq:spt-claim}
\spt(\textcolor{black}{\hat{T}_\infty})\subset Z_0=
\chi_\infty(X)=\zeta_j(X\times\{0\}\subset Z_{j,\infty})\subset Z.
\ee
Take any $z\in Z\setminus Z_0$
as in Figure~\ref{fig:SWIF-Tj}.
Then there exists $j_z\in {\mathbb N}$
such that $z=\zeta_{j_z}(f_{j_z}(x,t))$
where $t>0$. 
Then by the definition of
$d_Z$ in Proposition~\ref{prop:common-Z},
the ball:
\be
B_{d_Z}(z,t)\subset 
\zeta_{j_z}(f_{j_z}(X\times (0,h_{j_z}])))\subset \zeta_{j_z}(Z_{j_z})\setminus Z_0.
\ee
In particular 
\be
B_{d_Z}(z,t) \cap \zeta_{j}(f_j(X\times\{h_j\}))=\emptyset
\ee
for all $j>j_z$.
By the lower semicontinuity of mass
under weak convergence as in (\ref{eq:semi}) we have 
\be\label{eq:pf-semi}
||\textcolor{black}{\hat{T}_\infty}||(B_{d_Z}(z,t))
\le \liminf_{j\to\infty}
||T_j||(B_{d_Z}(z,t))=0
\ee
so $z\notin \spt(\textcolor{black}{\hat{T}_\infty})$.
Thus $\spt(\textcolor{black}{\hat{T}_\infty}) \subset Z_0$.   

Since $\chi_\infty:(X,d_\infty)\to (Z_0,d_Z)$ is an isometry, we can define:
\be\label{eq:S-is-defn}
\textcolor{black}{T_\infty}=\chi^{-1}_{\infty\#}\textcolor{black}{\hat{T}_\infty}
\ee
as an integral current on $(X,d_\infty)$.

We claim that for each $j\in {\mathbb N}$,  ``$\textcolor{black}{T_\infty}=T$ viewed
as an integral current on $(X,d_j)$'' in the sense of Definition~\ref{defn:equal-current}.  Given
any tuple, $(\pi_0, \pi_1,...,\pi_m)$ 
of Lipschitz functions on $(X, d_j)$,
we need only show 
\be\label{eq:need-T=S}
T(\pi_0, \pi_1,...,\pi_m)=\textcolor{black}{T_\infty}(\pi_0, \pi_1,...,\pi_m).
\ee 
First note that by Lemma~\ref{lem:Lip} these $\pi_i$
are also tuples of Lipschitz functions on $(X, d_\infty)$, so 
$\textcolor{black}{T_\infty}(\pi_0,...,\pi_m)$ is well 
defined.   

 \begin{figure}[h] 
   \centering \quad \includegraphics[width=.5\textwidth]{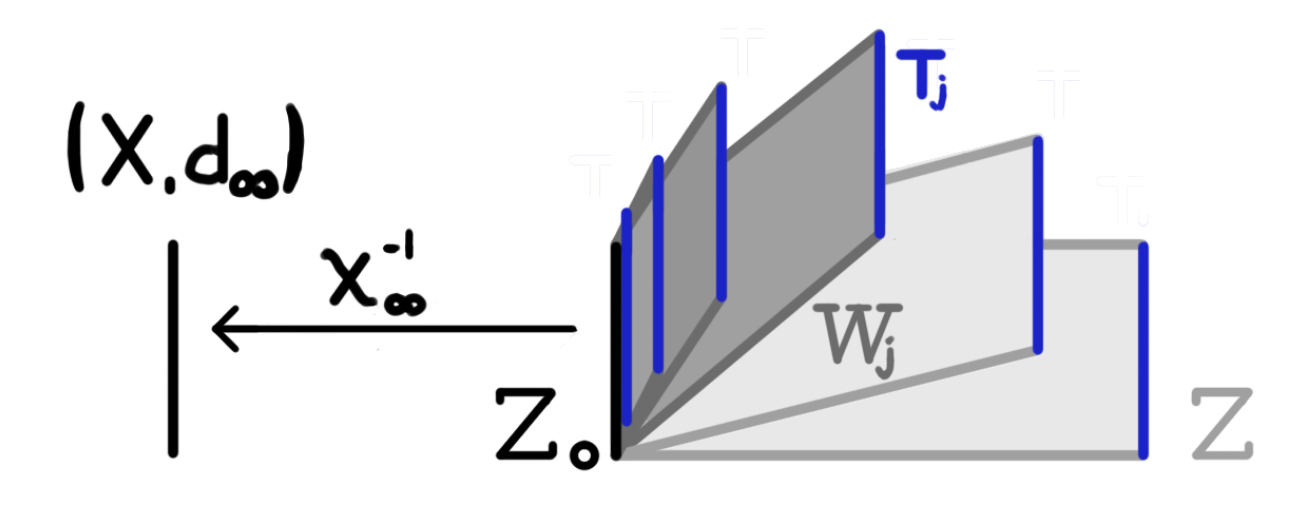} 
\caption{Here we see $W_j\subset Z$ of Lemma~\ref{lem:W-to-X} in dark gray, the rest of $Z$ is lighter, and
$Z_0\subset W_j\subset Z$ is in black with the
isometry, $\chi^{-1}_\infty$, to $(X,d_\infty)$.
}
\label{fig:Wj}
\end{figure}

Recall the set $W_j\subset Z$ from
Lemma~\ref{lem:W-to-X}, Lemma~\ref{lem:Lip-Z}, and Remark~\ref{rmrk:Lip-Z} depicted in Figure~\ref{fig:Wj}.
By Lemma~\ref{lem:Lip-Z}, we
define a tuple of Lipschitz functions,
$(\tilde{\pi}_0, \tilde{\pi}_1,...,\tilde{\pi}_m)$,
on the subset, $(W_j,d_Z)$, inside
$(Z,d_Z)$, 
such that
\be\label{eq:tuple-tilde}
\tilde{\pi}_i(\zeta_{j_k}(x,t))=\pi_i(x)
\qquad \forall x\in X,\,\, t\in [0,h_{j_k}],\,\,j_k\ge j.
\ee
By (\ref{eq:S-is-defn}) and
the fact that
\be 
\tilde{\pi}_i(\zeta_{j_k}(x,t))=\pi_i(x)=
\tilde{\pi}_i(\chi_\infty(x))) 
\qquad \forall x\in X,\,\, t\in [0,h_{j_k}],\,\,j_k\ge j,
\ee
we have
by definition of $\textcolor{black}{T_\infty}$
and then the definition of pushforward that
\begin{eqnarray}
\textcolor{black}{T_\infty}(\pi_0, \pi_1,...,\pi_m)
&=&
(\chi^{-1}_{\infty\#}\textcolor{black}{\hat{T}_\infty})
(\tilde{\pi}_0\circ \chi_\infty, \tilde{\pi}_1\circ \chi_\infty,...,\tilde{\pi}_m\circ \chi_\infty)\\
&=&
\textcolor{black}{\hat{T}_\infty}(\tilde{\pi}_0, \tilde{\pi}_1,...,\tilde{\pi}_m).
\end{eqnarray}
Note that for any  $j_k\ge j$ we have
\be
\spt(T_{j_k})\subset \zeta_{j_k}(Z_{j_k}) \subset W_j\subset Z.
\ee
So $T_{j_k}$
converge weakly as currents to $\textcolor{black}{\hat{T}_\infty}$ in $(W_j,d_Z)$. 
This means, for any Lipschitz tuple, $\omega$, on $(W_j,d_Z)$, we have
\be \label{eq:pf-weak}
T_{j_k}(\omega)\to \textcolor{black}{\hat{T}_\infty}(\omega).
\ee
Applying this to $\omega=(\tilde{\pi}_0, \tilde{\pi}_1,...,\tilde{\pi}_m)$,
we have
\be\label{eq:S-is-limjk}
\textcolor{black}{T_\infty}(\pi_0, \pi_1,...,\pi_m)=
\lim_{k\to \infty} T_{j_k}(\tilde{\pi}_0, \tilde{\pi}_1,...,\tilde{\pi}_m).
\ee
On the other hand, by 
(\ref{eq:tuple-tilde}),
we have,
\be
\tilde{\pi}_i(\zeta_{j_k}(x,t))=\pi_i(x)=
\tilde{\pi}_i(\zeta_{j_k}(f_{j_k}(x))))
\quad \forall x\in X,\,\, t\in [0,h_{j_k}],\,\,j_k\ge j.
\ee
Thus
$T_{j_k}=\zeta_{j_k\#}f_{j_k\#}T$
and
we have
\begin{eqnarray}
T_{j_k}(\tilde{\pi}_0,...,\tilde{\pi}_m)
&=& (\zeta_{j_k\#}f_{j_k\#}T)
(\tilde{\pi}_0\circ \zeta_{j_k}\circ f_{j_k},...,\tilde{\pi}_m\circ \zeta_{j_k}\circ f_{j_k})\\
&=&
T(\pi_0,...,\pi_m).
\end{eqnarray}  
Substituting this
into (\ref{eq:S-is-limjk}), and taking the
limit $j_k\to \infty$, we reach our claim
in (\ref{eq:need-T=S}).
We conclude that ``$\textcolor{black}{T_\infty}=T$ viewed
as an integral current on $(X,d_j)$'' in the sense of
Definition~\ref{defn:equal-current}.

Now we claim $\spt_{d_Z}(\textcolor{black}{\hat{T}_\infty})=Z_0$. Since we know $\textcolor{black}{\hat{T}_\infty}=\chi_{\infty\#}\textcolor{black}{T_\infty}$, this is the
same as showing $\spt_{d_{\infty}}(\textcolor{black}{T_\infty})=X$. 
Given any $x\in X$ and any $r>0$
we have
\be
||\textcolor{black}{T_\infty}||_{d_\infty}(B_{d_\infty}(x,r))
\ge 
||\textcolor{black}{T_\infty}||_{d_j}(B_{d_\infty}(x,r))
\ee
by Lemma~\ref{lem:current}.
By our hypotheses,
we have uniform convergence of the distance functions, $|d_j-d_\infty|<\epsilon_j\to 0$,
so for all $j$ such that $\epsilon_j<r/2$
\be
B_{d_\infty}(x,r) \supset B_{d_j}(x,r/2).
\ee
Combining this with the above,
and then the fact that ``$\textcolor{black}{T_\infty}=T$ as integral currents on $(X,d_j)$'', we have
\be
||\textcolor{black}{T_\infty}||_{d_\infty}(B_{d_\infty}(x,r))\ge
||\textcolor{black}{T_\infty}||_{d_j}(B_{d_j}(x,r/2))=
||T||_{d_j}(B_{d_j}(x,r/2))>0.
\ee
Thus $\spt_{d_{\infty}}(\textcolor{black}{T_\infty})=X$
and $\spt_{d_Z}(\textcolor{black}{\hat{T}_\infty})=Z_0$. 

Finally, we claim the masses converge
as in (\ref{eq:mass-pres}).
By Ambrosio-Kirchheim's semicontinuity of mass
combined with the fact that
$\zeta_\infty$ is distance preserving
we have
\be\label{eq:squeeze-above}
\mass_{d_\infty}(\textcolor{black}{T_\infty})=\mass_{d_Z}(\textcolor{black}{\hat{T}_\infty})
\le \liminf_{k\to\infty}
\mass_{d_Z}(T_{j_k})=
\liminf_{k\to\infty}\mass_{d_{j_k}}(T).
\ee
Applying Lemma~\ref{lem:current}, 
and then the fact that ``$\textcolor{black}{T_\infty}=T$ as integral currents on $(X,d_{j_k})$'',
we have
\be\label{eq:squeeze-below}
\mass_{d_\infty}(\textcolor{black}{T_\infty})\ge 
\mass_{d_{j_k}}(\textcolor{black}{T_\infty})=\mass_{d_{j_k}}(T).
\ee
Combining (\ref{eq:squeeze-above}) and
(\ref{eq:squeeze-below}),
we have the claimed mass convergence for the subsequence
in (\ref{eq:mass-pres}).
\end{proof}


\subsection{Proving SWIF Convergence}
\label{sect:pf-SWIF}

In this section we complete the proof of Theorem~\ref{thm:SWIF-cmpct}.  Recall Definition~\ref{defn:SWIF} of the SWIF distance.  Our goal is to provide a constructive proof that
\be
d_{SWIF}((X,d_j,T),(M_\infty,d_\infty,T_\infty))\to 0
\textrm{ as }j\to \infty
\ee
where $(M_\infty,d_\infty,T_\infty)$
is the integral current space defined using a subsequence in Proposition~\ref{prop:Tinfty}.

To complete the proof of Theorem~\ref{thm:SWIF-cmpct}
we will explicitly estimate the SWIF distances.  First we prove a more general proposition that provides a more general estimate on the SWIF distance between a pair of integral current spaces:

\begin{prop}\label{prop:TaTb}
Given two metric spaces,
$(X,d_a)$ and $(X,d_b)$,
and an $\epsilon>0$ such that
\be \label{eq:ab-Zab-2}
d_b(x,y)-\epsilon \le d_a(x,y) \le d_b(x,y)\quad \forall x,y\in X.
\ee
If $T_a$ is an integral current on $(X,d_a)$ and $T_b$ is an integral current on $(X,d_b)$ 
both of dimension, $m$, and ``$T_a=T_b$ as integral currents on $(X,d_a)$'', then
\be
d_{SWIF}((M_a,d_a,T_a),
(M_b, d_b, T_b)
\le h 2^{\frac{m+1}{2}}(
\mass_{d_b}(T_b)
+  \mass_{d_b}(\partial T_b))
\ee
where $M_a=\set(T_a)$, $M_b=\set(T_b)$, and $h=\epsilon/2$.
\end{prop}

\begin{proof}
Note that these hypotheses
imply the
hypotheses of Proposition~\ref{prop:Zab}; 
so
there exists a metric
space 
\be
(Z_{a,b}=X\times [0,h], d_{Z_{a,b}})
\textrm{
where } h=\epsilon/2,
\ee
with
\be \label{eq:d_Z-2}
d_{Z_{a,b}}(z_1,z_2)
\le d_{taxi}(z_1,z_2)=
d_b(x_1,x_2)+|t_1-t_2|,
\ee
and
there are distance preserving maps:
\begin{eqnarray}
&f_a:(X, d_a) \to Z_{a,b}&
\textrm{ s.t. } f_a(x)=(x,h),\\
&f_b:(X, d_b) \to Z_{a,b}&
\textrm{ s.t. } f_b(x)=(x,0).
\end{eqnarray}

We will estimate the SWIF distance
by constructing integral currents, $A$ and $B$, on $Z_{a,b}=X\times [0,h]$ such that
\be
\partial B + A= f_{a\#}T_a-f_{b\#}T_b 
\ee
and then bounding the masses of $A$ and $B$.   

We have an integral current space, $(M,d_b,T_b)$.  So the integral current $T_b\times [0,h]$ is defined on the isometric product space, $(M_b\times [0,h], d_{isom})$
defined in
Proposition 3.7 of \cite{PS-properties}  by Portegies-Sormani.   This proposition states
that
\be \label{eq:partialT1} 
\partial(T_b\times [0,h])=-\partial T_b \times [[(0,h)]]
+ T_b \times \partial [[(0,h)]],
\ee
and
\be \label{eq:partialT2} 
T_b \times \partial [[(0,h)]]=\psi_{h\#}T_b-\psi_{0\#}T_b
\ee
where
$\psi_t(x)=(x,t)\in X\times [0,h]$
and
\be
\mass_{d_{isom}}(T_b\times [0,h])=h \,\mass(T_b),
\ee
where $d_{isom}$ is the standard isometric product:
\be
d_{isom}((x_1,t_1),(x_2,t_2))=\sqrt{(d_b(x_1,x_2))^2+(t_1-t_2)^2}.
\ee

We have
$
M_{b}\times [0,h] \subset X\times [0,h]
=Z_{a,b}
$,
however
$d_{Z_{a,b}})$ is not the standard isometric product distance.
So we rescale the isometric product, compare to the taxi product, and apply (\ref{eq:d_Z-2}) to see that:
\be
\sqrt{2} \, d_{isom}(z_1,z_2)\ge d_{taxi}(z_1,z_2)\ge d_{Z_{a,b}}(z_1,z_2).
\ee
Note that $T_b\times [0,h]$ is an integral current on the
rescaled isometric product
space, $(X\times [0,h], \sqrt{2}d_{isom})$,
satisfying (\ref{eq:partialT1})-(\ref{eq:partialT2})
except now the mass rescales:
\be
\mass_{\sqrt{2} d_{isom}}(T_b\times [0,h])=2^{\frac{m+1}{2}}\, h \mass_{d_b}(T_b). 
\ee
We can apply the Lipschitz one identity map, 
\be
\iota: (X\times [0,h], \sqrt{2}d_{isom})\to (X\times [0,h], \sqrt{2}d_{Z_{a,b}}),
\ee
to define the pushforward integral current on $Z_{a,b}$:
\be
B=\iota_{\#}(T_b \times  [[(0,h)]])
\ee
whose mass satisfies
\be
\mass_{d_{Z_{a,b}}}(B) 
\le 2^{\frac{m+1}{2}}\, h \,\mass_{d_b}(T_b).
\ee
Similarly, we define
\be
A=\iota_{\#}(\partial T_b \times [[(0,h)]])
\ee
whose mass satisfies
\be
\mass_{d_{Z_{a,b}}}(A) 
\le 2^{\frac{m}{2}}\, h \,\mass_{d_b}(\partial T_b).
\ee
By the pushforward of (\ref{eq:partialT1}) we have
\be \label{eq:partialT3} 
\partial B=-A + \iota_{\#}(\,T_b \times \partial [[(0,h)]]\,)
\ee
and by the pushforward of (\ref{eq:partialT2}) we have
\be \label{eq:partialT4} 
\iota_{\#}(\,T_b \times \partial [[(0,h)]]\,)=
\iota_{\#}\psi_{h\#}T_b-\iota_{\#}\psi_{0\#}T_b.
\ee
Thus we need only show that
\be \label{eq:faTa}
f_{a\#}T_a=\iota_{\#}\psi_{h\#}T_b
\ee
and
\be
f_{b\#}T_b=\iota_{\#}\psi_{0\#}T_b
\ee
as integral currents on $(Z_{a,b},d_{Z_{a,b}})$.
The latter easily follows from,
\be
f_b(x)=(x,0)=\iota(x,0)=\iota(\psi_0(x)).
\ee
We also have
\be
f_a(x)=(x,h)=\iota(x,h)=\iota(\psi_h(x)).
\ee
which gives
\be 
f_{a\#}T_a=\iota_{\#}\psi_{h\#}T_a.
\ee
So we need only prove
\be \label{eq:psiTaTb}
\iota_{\#}\psi_{h\#}T_a=\iota_{\#}\psi_{h\#}T_b
\ee
as integral currents on $(Z_{a,b},d_{Z_{a,b}})$.

Consider any Lipschitz tuple $(\pi_0,...,\pi_{m})$
on $(Z_{a,b},d_{Z_{a,b}})$. Recall that
\be
\iota_{\#}\psi_{h\#}T_a(\pi_0,...,\pi_{m})
=T_a(\pi_0\circ \iota\circ \psi_h,...,\pi_m\circ \iota\circ \psi_h).
\ee
For each Lipschitz function
$\pi_i:(Z_{a,b},d_{Z_{a,b}}) \to {\mathbb{R}}$ we
claim that 
\be \label{eq:actually}
\pi_i\circ \iota\circ \psi_h: (X,d_a) \to {\mathbb{R}}
\textrm{ is Lipschitz}.
\ee
This follows from the fact that $f_a$ is distance
preserving and so
\begin{eqnarray}
|\pi_i(\iota(\psi_h(x_1)))-\pi_i(\iota(\psi_h(x_2)))|
&=&|\pi_i(x_1,h)-\pi_i(x_2,h)|\\
&\le& K \,d_{Z_{a,b}}((x_1,h),(x_2,h))\\
&=& K \,d_{Z_{a,b}}(f_a(x_1),f_a(x_2))\\
&=&K\,d_a(x_1,x_2).
\end{eqnarray}
So
$(\pi_0\circ \iota\circ \psi_h,...,\pi_m\circ \iota\circ \psi_h)$ is
a Lipschitz tuple on $(X,d_a)$.
Combining this with the hypothesis that
``$T_a=T_b$ as integral currents on $(X,d_a)$'' we have
\be
T_a(\pi_0\circ \iota\circ \psi_h,...,\pi_m\circ \iota\circ \psi_h)=
T_b(\pi_0\circ \iota\circ \psi_h,...,\pi_m\circ \iota\circ \psi_h).
\ee
Thus, by the definition of pushforward, we have (\ref{eq:psiTaTb}).
\end{proof}

We now apply Proposition~\ref{prop:unif-cmpct}, Proposition~\ref{prop:Tinfty}, and
Proposition~\ref{prop:TaTb} to
prove Theorem~\ref{thm:SWIF-cmpct}:

\begin{proof}
First note that the hypotheses of
Theorem~\ref{thm:SWIF-cmpct} imply that
the hypotheses of Proposition~\ref{prop:unif-cmpct} and Proposition~\ref{prop:Tinfty}.
By Proposition~\ref{prop:unif-cmpct},
we have uniform convergence of
compact metric spaces
$(X, d_j)$ to a compact limit
$(X,d_\infty)$:
\be
d_\infty(x,y)-\epsilon_j \le d_j(x,y) \le d_\infty(x,y)\quad \forall x,y\in X
\ee
where $\epsilon_j \to 0$ as $j\to \infty.$

By Proposition~\ref{prop:Tinfty}
we have an integral current space,
\be
(M_\infty, d_\infty, T_\infty)
\textrm{ where }
M_\infty=\set(T_\infty)\subset X
\ee
with $Cl(M_\infty)=X$
and
``$T_\infty=T$ as integral currents on
$(X,d_j)$'' for all $j\in {\mathbb N}$. 

Next we apply Proposition~\ref{prop:TaTb}
with 
\be
(X,d_a,T_a)=(X,d_j,T) 
\textrm{ and }
(X,d_b,T_b)=(X,d_\infty,T_\infty)
\ee
where $\set(T_j)=X$ and $\set(T_\infty)=M_\infty$
and 
\be
h=h_j=2\epsilon_j \to 0 \textrm{ as }j \to \infty.
\ee
Thus by Proposition~\ref{prop:TaTb}, we have
\be
d_{SWIF}((X,d_j,T),(M_\infty, d_\infty,T_\infty))
\le 2^{\frac{m+1}{2}} h_j(\mass_{d_\infty}(T_\infty)+
\mass_{d_\infty}(\partial T_\infty)).
\ee
Thus by (\ref{eq:mass-Tj}) of Proposition~\ref{prop:Tinfty} we have,
\be
d_{SWIF}((X,d_j,T),(M_\infty, d_\infty,T_\infty))
\le 2^{\frac{m+1}{2}} \,h_j\,(V_0+A_0)\to 0
\ee
and we may conclude SWIF convergence of the
original sequence.

Finally we apply the fact that
$\mass_{d_j}(T)$
is monotone increasing by
Lemma~\ref{lem:current} and
the fact that
a subsequence converges to
$\mass_{d_\infty}(T_\infty)$ 
by Proposition~\ref{prop:Tinfty}
to
conclude that mass converges.
Thus we have volume preserving
intrinsic flat convergence of the original sequence.
\end{proof}

\section{Open Problems}
\label{sect:Open}

Let us consider a monotone increasing sequence of integral current spaces, $(X,d_j,T)$, with $d_j\le d_{j+1}$ with the following three hypotheses
\be\label{eq:Wenger-hyp}
\diam_{d_j}(X)\le D_0,
\quad 
\mass_{d_j}(T) \le V_0,
\quad
\mass_{d_j}(\partial T) \le A_0.
\ee
By Lemma~\ref{lem:diam}, we have a pointwise
limit, $(X,d_\infty)$.
We know that without the compactness assumption in our Theorem~\ref{thm:Riem}
and Theorem~\ref{thm:SWIF-cmpct},
that we may not have Gromov-Hausdorff Convergence of the sequence.  See Example~\ref{ex:not-unif}.

However, we know that by Wenger's Compactness Theorem of
\cite{Wenger-compactness} that any sequence of
integral current spaces satisfying (\ref{eq:Wenger-hyp})
has a
subsequence converging to a SWIF limit,
\be
(X,d_{j_k},T) \SWIFto
(X'_\infty,d'_\infty,T'_\infty),
\ee  
where the SWIF limit might be the zero space.

Combining the monotonicity assumptions with (\ref{eq:Wenger-hyp}), 
perhaps one can show one or any of the following:
\begin{itemize}
\item the sequence itself SWIF converges without needing a subsequence possibly to a nonzero limit space,
\item there is  $\mathcal{VF}$ convergence
to the limit space, 
\item
there is a distance preserving map from the
SWIF limit, $(X'_\infty, d'_\infty)$, into
the pointwise limit, $(X, d_\infty)$, 
with possibly some relationship between the currents, $T_\infty'$ and
$T$.
\end{itemize}
Note that assuming any of the following additional hypotheses might help:
\begin{itemize}
\item The integral current spaces, $(X,d_j,T)$, are
Riemannian, $(M,d_{g_j},[[M]])$,
possibly without boundary.
\item 
The integral current spaces, $(X,d_j,T)$, have empty boundary, $\partial T=0$, so
we don't have to use $A_j$ in the proof of SWIF convergence.
\item The integral current structure, $T$, is assumed
to be an integral current on the pointwise limit, $(X,d_\infty)$. 
\end{itemize}
It would be interesting to find counter examples to any or all of these statements as well.

\bibliographystyle{alpha}
\bibliography{PSbibliography}

\end{document}